\documentclass{article}

\usepackage[utf8]{inputenc}
\usepackage[scale=0.68]{geometry}

\usepackage{amsmath}
\usepackage{amsfonts}
\usepackage{amssymb}
\usepackage{amsthm}
\usepackage{mathtools}

\usepackage{braket}
\usepackage{enumerate} 
\usepackage{cancel}
\usepackage{dsfont} 

\usepackage{placeins} 

\usepackage{tabularx} 
\usepackage{caption}
\usepackage{subfig}
\usepackage{multirow}
\usepackage{booktabs}

\usepackage{algorithmicx}
\usepackage{algpseudocode}
\usepackage{algorithm}

\usepackage[backend=biber, 
            language=english,%
            style=alphabetic,%
            firstinits=true,
            maxbibnames=5 
           ]{biblatex}
\usepackage{hyperref}

\hypersetup{%
    colorlinks=false, hidelinks, bookmarksnumbered,
    hyperfootnotes=false,pdfpagelabels,pdfencoding=auto,
    bookmarksopen=true, bookmarksopenlevel=1, unicode,
    pdftitle={Reconstruction of a Time-dependent Potential from Wave Measurements},%
    pdfauthor={Thies Gerken, Armin Lechleiter},%
    pdfsubject={parameter identification},%
    pdfkeywords={dynamic inverse problems; parameter identification; inexact Newton regularization}
}

\newtheorem{theorem}{Theorem}[section]
\newtheorem{corollary}[theorem]{Corollary}
\newtheorem{lemma}[theorem]{Lemma}

\theoremstyle{definition}
\newtheorem{definition}[theorem]{Definition}
\newtheorem{example}[theorem]{Example}

\theoremstyle{remark}

\numberwithin{equation}{section}

\DeclareMathOperator{\diam}{diam}

\DeclareMathOperator{\lin}{lin}

\DeclareMathOperator*{\esssup}{ess\,sup}

\newcommand\R{\ensuremath{\mathbb R}}
\newcommand\N{\ensuremath{\mathbb N}}

\newcommand\map[3]{\ensuremath{{#1}\,\colon \,{#2}\to{#3}}}

\newcommand{\subnorm}[2]{{{\left \|#2\right \|}_{#1}}}
\newcommand{\norm}[1]{{{\subnorm{{}}{#1}}}}
\newcommand{\abs}[1]{\lvert#1\rvert}

\newcommand\dt{\ensuremath{{\,\mathrm{d}t}}}
\newcommand\ds{\ensuremath{{\,\mathrm{d}s}}}
\newcommand\dx{\ensuremath{{\,\mathrm{d}x}}}
\newcommand\dy{\ensuremath{{\,\mathrm{d}y}}}

\newcommand\dd{\ensuremath{{\,\mathrm{d}}}}

\newcommand\scp[2]{{\left({#1},{#2} \right)}}
\newcommand\dup[2]{{\langle{#1},{#2} \rangle}}
\newcommand{\pprime}{{\prime\prime}}
\newcommand{\laplace}{\Delta}
\newcommand{\ale}{a.e.}
\newcommand*{\lcdot}{\text{\raisebox{-0.4ex}{\scalebox{1.75}{$\cdot$}}}}%

\newcommand{\assign}{\coloneqq}
\newcommand{\assigns}{\eqqcolon}

\renewcommand{\phi}{\varphi}
\renewcommand{\epsilon}{\varepsilon}

\title{Reconstruction of a Time-dependent Potential\\ from Wave Measurements}

\author{Thies Gerken\thanks{RTG 2224 \enquote{Parameter Identification - Analysis, Algorithms, Applications} and Center for Industrial Mathematics, Universit\"at Bremen, Germany; \texttt{tgerken@math.uni-bremen.de}} \and Armin Lechleiter\thanks{Center for Industrial Mathematics, Universit\"at Bremen, Germany; \texttt{lechleiter@uni-bremen.de}}}

\date{\today}

\begin{document}

\maketitle

\begin{abstract}
  We add a time-dependent potential to the inhomogeneous wave equation and consider the task of reconstructing this potential from measurements of the wave field. This dynamic inverse problem becomes more involved compared to static parameters, as, e.g.~the dimensions of the parameter space do considerably increase.
  We give a specifically tailored existence and uniqueness result for the wave equation and compute the Fr\'echet derivative of the solution operator, for which also show the tangential cone condition.
  These results motivate the numerical reconstruction of the potential via successive linearization and regularized Newton-like methods.
  We present several numerical examples showing feasibility, reconstruction quality, and time efficiency of the resulting algorithm.
\end{abstract}


\section{Introduction}

We consider the inhomogeneous wave equation in a bounded time-space domain $[0,T] \times \Omega$ with a time- and space-dependent potential $c$ and a source $f$,
\begin{equation}\label{eq:wavEqu}
  u^\pprime - \Delta u + cu = f \quad \text{in } [0,T] \times \Omega.
\end{equation}
For this setting, we tackle the dynamic inverse problem to reconstruct $c$ from measurements of $u$ at specific measurement points for many time steps.
This inverse problem provides a simplified model for the non-destructive testing via time-dependent waves in dynamic environments of, e.g.~complex carbon-fiber-reinforced polymers under loadings; it is additionally crucial for the detection of non-linear terms in the wave equations from merely an approximate linear model.

Our aim is to show that this dynamic inverse problem for the time-varying quantity $c$ can be mathematically rigorously formulated, analyzed, and stably solved by successive linearization in reasonable computation time.
Thus, we first construct suitable function spaces for the coefficients and the solutions to solve the latter partial differential equation with homogeneous initial and boundary conditions in $n=1$, $2$, or $3$ dimensions.
Then we show that the parameter-to-solution map is Fr\'echet differentiable on a suitable domain of definition.

The inverse problem to determine $c$ from full measurements of $u$, as well as its linearization, both turn out to be ill-posed. As we can however show that for our setting the tangential cone condition of Scherzer~\cite{Scherzer95} is satisfied, we consider successive linearization as the starting point for an inversion algorithm.

To be able to cope with the more important case of reduced point measurements of the wave field, too, we compute the necessary operator adjoints and finally detail several numerical experiments computed by the so-called REGINN algorithm of Rieder~\cite{rieder:cgreginn} when applied to the inverse problem.
Roughly speaking, these experiments show that the observable region in space is determined by the excitations and the sensor positions up to errors due to the noise level.

There are not so many papers in the literature tackling inverse problems for space- and time-dependent parameters of a wave equation.
On the theoretical side, there are a couple of papers proving uniqueness result for various types of coefficients and data, see, e.g.~\cite{Stefanov1989, Ramm1991}, together with several more recent works that particularly indicate the rising interest in the topic, see~\cite{Kian2016, Aicha2015, Salazar2013, Eskin2007}.
The main tool of many of these papers are geometric optics solutions.
Concerning numerical algorithms, there does not seem to exist a similar variety of results, apart from the detection of time-dependent (point) sources for the wave equation, see, e.g.~\cite{Badia2001}.
We would like to further note reconstruction results for non-linear elastic materials in~\cite{Binder2015} indicating future potential fields of application for the algorithms from this paper.

Our solution theory for~\eqref{eq:wavEqu} follows the weak solution theory of Lions und Magenes~\cite{lionsmagenes:bvp1} as the latter can also be used for more complicated problems than considered in this paper.
This weak solution theory shows existence and uniqueness of solution to~\eqref{eq:wavEqu} for all $c\in H^2([0,T], L^2(\Omega))$ that are bounded from below by some (possibly negative) $c_0\in \R$ and all square-integrable $f$.

A similar framework provided in Evans' book~\cite[Chapter 7.2]{evans:pde} requires $c\in C^1([0,T]\times \overline{\Omega})$.
If one aims to embed the latter space in some Hilbert- or at least in some reflexive Banach space, one typically ends up in a high-order Sobolev space.
Thus, resulting reconstructed parameters will then typically possess some extra spatial smoothness, which is somewhat inconvenient from the point of view of applications.

The above-mentioned set of suitable parameters for our solution theory is not yet ready to show, e.g.~Fr\'echet differentiability of the solution operator $c \mapsto u$, but by well-known analytic tools we prove that a suitable open domain for such a derivative exists.

Of course, more general settings than~\eqref{eq:wavEqu} should include more general variable coefficients, which is however out of this paper's scope.
Similarly, regularization methods that can be rigorously formulated in Banach spaces to the inversion problem under consideration is a natural continuation of this work that will be considered in a future work.

The remainder of this paper is structured as follows:
Sections~\ref{sec:solutiontheory} and~\ref{sec:frechet} treat weak solutions and Fr\'echet derivatives for the wave equation.
Ill-posedness of the resulting inverse problems is shown in Section~\ref{sec:illposed}.
Sections~\ref{sec:discWave} and~\ref{sec:adjoints} consider the discretization of all operators and their adjoints that are involved in the setting.
Finally, Section~\ref{sec:reginn} details the inversion algorithm and Section~\ref{sec:numerics} presents numerical examples.

\emph{Notation:} If there is no danger of confusion, we write $L^p$ instead of $L^p(\Omega)$, and analogously $H^k$ instead of $H^k(\Omega)$ for Sobolev spaces; corresponding spaces of functions with zero traces get an additional index $0$.
Further, $\dup{\cdot}{\cdot}$ is the duality product between $H^{-1}\assign (H^1_0){\vphantom{1pt}}^\prime$ and $H^1_0$ and $\scp{\cdot}{\cdot}$ is the scalar product of $L^2$.
We identify $L^2$ with its dual space such that we typically work in the Gelfand triple $H^1_0\subset L^2 \subset H^{-1}$.
We do not distinguish between scalar and vector-valued functions and, e.g.~also write $\scp{\nabla u}{\nabla v}$ for the ${L^2(\Omega)}^n$-scalar product of $\nabla u$ and~$\nabla v$.

\section{Existence of solution to the wave equation for time-de\-pend\-ent pa\-ra\-me\-ters}
\label{sec:solutiontheory}

In this section, we show existence theory for the wave equation with time-dependent parameter $c=c(t,x)$ with $t \in [0,T]$ for some $T>0$ and $x$ in a bounded Lipschitz domain $\Omega\subset \R^n$ for $n\in \{1,2,3\}$.
The system excitation is modeled by a source $f\in L^2([0,T] \times \Omega)$, such that the initial boundary value problem for the wave field $u$ reads
\begin{equation}\label{eq:arwp}
  \begin{cases}
    u^\pprime(t,x) - \laplace u(t,x) + c(t,x)u(t,x) = f(t,x)\qquad & (t,x) \in (0,T) \times \Omega, \\
    u(0,x) = u^\prime(0,x) = 0 \qquad & x\in \Omega, \\
    u(t,x) = 0 \qquad & (t,x)\in (0,T)\times\partial \Omega.
  \end{cases}
\end{equation}
We initially require $c\in L^2([0,T], L^2)$, but we finally will require more regularity of this parameter.
Multiplication of the wave equation in~\eqref{eq:arwp} by a space-dependent test function $\phi\in H^1_0$, integration over $\Omega$, and partial integration yields the following definition of a weak solution to~\eqref{eq:arwp}.

\begin{definition}\label{def:schwacheloesung}
  A function $u\in H^2([0,T], H^{-1}) \cap H^1([0,T], L^2) \cap L^2([0, T], H^1_0)$ is a weak solution to~\eqref{eq:arwp} if
  \begin{equation}\label{eq:theo:schwachelsgalt}
  \dup{u^\pprime(t)}{\phi} + \int_{\Omega} \left[ \nabla u(t) \cdot \nabla \phi + c(t) u(t) \phi \right] \dx = \scp{f(t)}{\phi}
  \end{equation}
  holds for almost every (a.e.) $t\in[0,T]$ and all $\phi\in H^1_0$, and if $u$ satisfies the initial conditions $u(0)=0$ in $L^2$ and $u^\prime(0) = 0$ in $H^{-1}$.
\end{definition}

Note that the integral in~\eqref{eq:theo:schwachelsgalt} involving $c$ is well defined by the smoothness of both $u(t)$ and $\phi(t) \in H^1_0$:
For $n\leq 3$ the embedding $H^1_0 \hookrightarrow L^4$ is continuous, such that the integrand is at least in $L^1$.
Further, the initial conditions for $u(0)$ and $u^\prime(0)$ are well defined because $u \in H^1([0,T], L^2) \hookrightarrow C([0,T], L^2)$
and $u^\prime \in H^1([0,T], H^{-1}) \hookrightarrow C([0,T], H^{-1})$, such that $u(0)$ naturally belongs to $L^2$ and $u^\prime(0)\in H^{-1}$.

If we set
\[
\map a{[0,T]\times H^1_0 \times H^1_0}{\R},\,\quad   a(t, u, v) \assign \int_{\Omega} \left[ \nabla u \cdot \nabla v + c(t) u v \right] \dx,
\]
then the weak formulation from the last Definition~\ref{def:schwacheloesung} is equivalent to
\begin{equation}\label{eq:theo:schwachelsg}
  \dup{u^\pprime(t)}{\phi} + a(t, u(t), \phi) = \scp{f(t)}{\phi} \quad \text{for all } \phi\in H^1_0 \text{ and \ale\ $t\in[0,T]$}.
\end{equation}
To construct such a weak solution, we proceed by Galerkin approximation in finite-di\-men\-si\-o\-nal subspaces of $H^1_0$.
To this end, choose some orthogonal basis $(\phi_j)_{j\in\N}$ of $H^1_0$ that is at the same time an orthonormal basis of $L^2$ (e.g.~via the eigenfunctions of the Laplacian).
Working with the Gelfand triple $H^1_0 \subset L^2 \subset H^{-1}$, the equalities $\dup{\phi_i}{\phi_j}=\delta_{ij}$ imply that the $\phi_j$ are also dual and normalized to each other for the duality product between $H^{-1}$ and $H^1_0$.

Plugging the finite-dimensional ansatz $u_m(t) \assign \sum_{j=1}^m \alpha_j(t) \phi_j
\in \lin \{\phi_j \,|\, j=1,\dots, m\}$ for some $m\in\N$ into~\eqref{eq:theo:schwachelsg}, we note that $u_m$ needs to solve
\begin{equation}\label{eq:theo:galerkin}
  \dup{u_m^\pprime(t)}{\phi} + a(t, u_m(t), \phi) = \scp{f(t)}{\phi} \quad \text{for all } \phi\in \lin \{\phi_j \, |\, j=1,\dots, m \}
\end{equation}
with zero initial conditions.
As is well-known for parameters that are constant in time, this yields $m$ ordinary differential equations for the coefficients $\alpha_j$ for $j=1,\dots,m$ with right-hand sides $t \mapsto \scp{f(t)}{\phi_i}$ in $L^2([0,T])$ that are all uniquely solvable in $H^2([0,T])$.

\begin{lemma}
  For $c\in L^2([0, T], L^2)$ and $m\in\N$ there is $u_m\in H^2([0,T], H^1_0(\Omega))$ with zero initial conditions that solves~\eqref{eq:theo:galerkin} for a.e.~$t\in[0,T]$.
\end{lemma}

We next compute an explicit constant that bounds the norms of all $u_m$.
As $u_m(t)$ belongs to the dense subset $\lin \{\phi_j\, |\, j\in\N \}$ of $H^1_0(\Omega)$, this will finally show that $u_m$ converges to a weak solution of the wave equation.

\begin{lemma}[Energy estimates]\label{lemma:energie}
  For $c\in H^1([0,T], L^2)$ with $c^\prime \in L^\infty([0,T], L^2)$ we assume that there is $c_b \in L^\infty([0,T], L^2)$ with $c_b\geq c_0\in \R$ such that
  \begin{equation}\label{eq:constantDelta}
    \norm{c-c_b}_{L^\infty([0,T], L^2)} < \delta(\Omega) \assign \left(2(1+C_{\mathrm{P}}^2)C^2_{H^1_0\hookrightarrow L^4}\right)^{-1},
  \end{equation}
  where $C_{H^1_0\hookrightarrow L^4}$ is the operator norm of the embedding $H^1_0(\Omega)\hookrightarrow L^4(\Omega)$ and $C_{\mathrm{P}}$ denotes the Poincar\'e constant of $H^1_0(\Omega)$. Then $u_m$ from~\eqref{eq:theo:galerkin} satisfies for all $m\in\N$ that
  \begin{align*}
    \esssup_{t\in[0,T]}\norm{u_m(t)}_{H^1_0} + &\esssup_{t\in[0,T]}\norm{u_m^\prime(t)}_{L^2} +
    \norm{u_m^\pprime}_{L^2([0,T], H^{-1})} \\
    &\leq \left(1+\norm{c}_{L^2([0,T], L^2)}\right) e^{C_1 \left(1+\norm{c^\prime}_{L^\infty([0,T], L^2)}\right)}\norm{f}_{L^2([0,T], L^2)}
  \end{align*}
  with $C_1>0$ depending only on $T$, $\Omega$ and $c_0$.
  If, additionally, $c\in H^2([0,T], L^2)$, then
  \begin{equation*}
    \esssup_{t\in[0,T]}\norm{u_m(t)}_{H^1_0} + \esssup_{t\in[0,T]}\norm{u_m^\prime(t)}_{L^2} +
    \norm{u_m^\pprime}_{L^2([0,T], H^{-1})} \leq e^{C_2 \left(1+\norm{c}_{H^2([0,T], L^2)}\right)}\norm{f}_{L^2([0,T], L^2)},
  \end{equation*}
  where $C_{1,2}>0$ merely differ by a constant depending on $\Omega$ and $T$.
\end{lemma}

The energy estimate depends on the lower bound $c_0 \in \R$ of the comparison parameter $c_b$.
Even if $c_0$ is completely arbitrary, we \textit{fix} this constant from now on to avoid technicalities.

\begin{proof}
  (1) As $u_m^\prime(t) = \sum_{j=1}^m \alpha_j'(t) \phi_j$ belongs to $\lin \{\phi_j\, |\, j\in\N \}\subset H^1_0$, we plug $u_m^\prime$ as test function into the weak formulation of $u_m$,
  \begin{equation}\label{eq:theo:energiebeweis}
    \dup{u_m^\pprime(t)}{u_m^\prime(t)} + a(t, u_m(t), u_m^\prime(t)) = \scp{f(t)}{u_m^\prime(t)}
    \quad \text{for a.e.~$t\in [0,T]$}.
  \end{equation}
  Due to the representation of $u_m^\prime(t)$, we directly get that
  \[
    \dup{u_m^\pprime(t)}{u_m^\prime(t)} = \sum_{i=1}^m \alpha_i^\pprime(t) \alpha_i^\prime(t) = \frac12 \frac{\dd}{\dt} \sum_{i=1}^m \alpha_i^\prime(t)^2 = \frac12 \frac{\dd}{\dt} \norm{u_m^\prime(t)}^2_{L^2},
  \]
  and, analogously, $\scp{\nabla u_m^\prime(t)}{\nabla u_m(t)} = \tfrac{\dd}{\dt} \norm{\nabla u_m}^2_{L^2(\Omega)^n} /2$.
  Setting
  \[
    a^\prime(t, v, w) \assign \frac \dd\dt a(t, v, w) = \int_\Omega c^\prime(t) v w \dx
    \quad \text{for } v, w\in H^1_0,
  \]
  one computes, again by finiteness of the sum representation of $u_m^\prime$, that
  \begin{align}
    \frac \dd \dt a(t, u_m(t), u_m(t)) &= \frac \dd \dt \left(\norm{\nabla u_m(t)}^2_{L^2(\Omega)^n} + \int_{\Omega} c(t) u_m(t)^2 \dx \right) \notag \\
    &= 2 \, a(t, u_m(t), u_m^\prime(t)) + a^\prime(t, u_m(t), u_m(t)). \label{eq:theo:energiebeweis3}
  \end{align}
  Together with~\eqref{eq:theo:energiebeweis}, the latter equality shows that for a.e.~$t\in[0,T]$ there holds
  \[
    2\,\scp{f(t)}{u_m^\prime(t)} + a^\prime(t, u_m(t), u_m(t)) = \frac \dd \dt a(t, u_m(t), u_m(t)) + \frac{\dd}{\dt} \norm{u_m^\prime(t)}^2_{L^2}.
  \]
  Thus, the fundamental theorem of analysis and the zero initial conditions for $u_m$ imply that
  \begin{equation}\label{eq:aux215}
    \int_0^t \big[ 2\,\scp{f(s)}{u_m^\prime(s)} + a^\prime(s, u_m(s), u_m(s)) \big] \ds
    = a(t, u_m(t), u_m(t)) + \norm{u_m^\prime(t)}^2_{L^2}.
  \end{equation}
  Bounding the right-hand side from below by Poincar\'e's estimate $\norm{u_m}_{L^2}\leq C_{\mathrm{P}} \norm{\nabla u_m}_{L^2}$ we hence arrive at
  \begin{align*}
    a(t, u_m(t), u_m(t)) \ + \ & \norm{u_m^\prime(t)}^2_{L^2} = \norm{\nabla u_m(t)}^2_{L^2} + \int_\Omega c(t) \abs{u_m(t)}^2 \dx + \norm{u_m^\prime(t)}^2_{L^2} \notag \\
    &\geq (1+C_{\mathrm{P}}^2)^{-1}\, \norm{ u_m(t)}^2_{H^1_0} + \int_\Omega c(t) \abs{u_m(t)}^2 \dx + \norm{u_m^\prime(t)}^2_{L^2}.
  \end{align*}
  (2) As $c$ may take negative values, it is now crucial to bound the second integral on the right via the auxiliary function $c_b(t)$ that is by assumption bounded from below by $c_0 \in \R$,
  \begin{align*}
    \int_\Omega c(t) u_m(t)^2 \dx &= \int_\Omega \big[ c_b(t) \abs{u_m(t)}^2 + (c(t)-c_b(t)) \abs{u_m(t)}^2 \big] \dx \\
    &\geq c_0 \norm{u_m(t)}^2_{L^2} - \left(2(1+C_{\mathrm{P}}^2)C^2_{H^1_0\hookrightarrow L^4}\right)^{-1} \norm{u_m(t)}^2_{L^4} \\
    &\geq c_0 \norm{u_m(t)}^2_{L^2} - \frac{1}{2(1+C_{\mathrm{P}}^2)} \, \norm{u_m(t)}^2_{H^1_0}.
  \end{align*}
  As the constant in front of $\norm{u_m(t)}^2_{H^1_0}$ is positive, we conclude that
  \[
    a(t, u_m(t), u_m(t)) + \norm{u_m^\prime(t)}^2_{L^2} \geq \frac{1}{2(1+C_{\mathrm{P}}^2)}\, \norm{ u_m(t)}^2_{H^1_0} + c_0 \norm{u_m(t)}_{L^2}^2 + \norm{u_m^\prime(t)}^2_{L^2}.
  \]
  In combination with~\eqref{eq:aux215}, we have hence shown that
  \begin{align*}
    & \frac{1}{2(1+C_{\mathrm{P}}^2)}\norm{ u_m(t)}^2_{H^1_0} + \norm{u_m^\prime(t)}^2_{L^2} + c_0 \norm{u_m(t)}_{L^2}^2 \\
    & \quad \leq \norm{f}^2_{L^2([0,T], L^2)} + \int_0^t \left[\norm{u_m^\prime(s)}^2_{L^2} + \norm{c^\prime}_{L^\infty([0,T], L^2)} C^2_{H^1_0\hookrightarrow L^4}\norm{u_m(s)}^2_{H^1_0} \right] \ds.
  \end{align*}
  (3) As we aim to apply Gronwall's inequality (see, e.g.~\cite[Appendix B.2]{evans:pde}), we need to control the term $c_0 \norm{u_m(t)}_{L^2}^2$. This is simple if $c_0$ is non-negative, as that term can then be dropped.
  More generally,
  \begin{align*}
    c_0 \norm{u_m(t)}_{L^2}^2 &= c_0 \int_\Omega \abs{u_m(t)}^2 \dx = c_0 \int_\Omega {\left(\int_0^t u_m^\prime(s) \ds \right)}^2 \dx \\
    &\geq \min(0, c_0) \int_\Omega t \int_0^t \abs{u_m^\prime(s)}^2 \ds \dx \geq \min(0, T  c_0) \int_0^t \norm{u_m^\prime(s)}_{L^2}^2 \ds,
  \end{align*}
  such that
  \begin{align*}
    \frac{1}{2(1+C_{\mathrm{P}}^2)}\,\left( \norm{u_m(t)}^2_{H^1_0} + \norm{u_m^\prime(t)}^2_{L^2}\right)
    \leq &\max\left \{1,1 -T  c_0, C^2_{H^1_0\hookrightarrow L^4}\norm{c^\prime}_{L^\infty([0,T], L^2)}\right \} \\
    & \times \left(\norm{f}^2_{L^2([0,T], L^2)} \! \! + \int_0^t \norm{u_m(s)}^2_{H^1_0} + \norm{u_m^\prime(s)}^2_{L^2}  \ds\right).
  \end{align*}
  Now, Gronwall's inequality implies that
  \begin{equation}\label{eq:theo:energiebeweis2}
    \norm{u_m(t)}^2_{H^1_0} + \norm{u_m^\prime(t)}^2_{L^2}
    \leq \norm{f}^2_{L^2([0,T], L^2)} \exp\left(C(T, c_0, \Omega) \left(1+\norm{c^\prime}_{L^\infty([0,T], L^2)}\right)\right)
  \end{equation}
  where $C(T,c_0,\Omega)$ is a placeholder for a constant depending only on $T,c_0$ and $\Omega$.\hfill \\
  (4) To obtain $H^{-1}$-bounds for $u_m^\pprime(t)$, let us finally choose any $v\in H^1_0(\Omega)$ with $\norm{v}_{H^1_0} = 1$
 and note that for a.e.~$t\in[0,T]$,
  \begin{align*}
    \dup{u_m^\pprime(t)}{v} &= \scp{u_m^\pprime(t)}{v}
    = \scp{f(t)}{v} - a(t, u_m(t), v) \\
    &\leq \norm{f(t)}_{L^2} \norm{v}_{L^2} + \norm{\nabla u_m(t)}_{L^2} \norm{\nabla v}_{L^2} + \norm{c(t)}_{L^2} \norm{u_m(t)}_{L^4}\norm{v}_{L^4} \\
    &\leq \norm{f(t)}_{L^2}  + \left(1+C^2_{H^1_0\hookrightarrow L^4}\norm{c(t)}_{L^2}\right) \norm{u_m(t)}_{H^1_0}.
  \end{align*}
  Representing the $H^{-1}$-norm as a dual norm shows that
  \begin{align*}
    &\norm{u_m^\pprime}^2_{L^2([0,T], H^{-1})} = \int_0^T \sup_{\| v \|_{H^1_0}=1} |\dup{u_m^\pprime(t)}{v}|^2 \dt \\
    &\ \leq \int_0^T \left(\norm{f(t)}_{L^2}  + \left(1+C^2_{H^1_0\hookrightarrow L^4}\norm{c(t)}_{L^2}\right) \norm{u_m(t)}_{H^1_0} \right)^2 \dt \\
    &\ \leq \left(1+\norm{c}_{L^2([0,T], L^2)}\right)^2 \exp\left(C(T, c_0, \Omega) \left(1+\norm{c^\prime}_{L^\infty([0,T], L^2)}\right)\right)\norm{f}^2_{L^2([0,T], L^2)}.
  \end{align*}
  Together with~\eqref{eq:theo:energiebeweis2} this shows the lemma's first claimed bound.
  For $c\in H^2([0,T], L^2)$ the estimate simplifies as $c^\prime \in H^1([0,T], L^2)$, a space which is continuously embedded in $L^\infty([0,T], L^2)$.
  The arising norms of $c$ und $c^\prime$ can hence be bounded by $\norm{c}_{H^2([0,T], L^2)}$. \qedhere
\end{proof}

As is well-known, the estimate of the last lemma cannot be shown transferred to the (weak) limit $u$ of the bounded sequence $u_m$, because the regularity of $u(t)$ is too low to test its variational formulation by $u^\prime(t)$.
The energy estimates, however, do allow to prove existence of at least one solution to~\eqref{eq:theo:schwachelsg}.
Uniqueness of this solution is then achieved by a standard regularity trick.

\begin{theorem}\label{satz:existenz}
  Under the assumptions of Lemma~\ref{lemma:energie} there is a unique weak solution $u$ to~\eqref{eq:arwp} that satisfies the energy estimates of Lemma~\ref{lemma:energie}.
\end{theorem}

\begin{proof}
  Concerning existence we note that the sequence $(u_m)_{m\in\N}$ is bounded in $L^2([0, T], H^1_0)$,
  $(u_m^\prime)_{m\in\N}$ is bounded in $L^2([0, T], L^2)$ and $(u_m^\pprime)_{m\in\N}$ is bounded in $L^2([0, T], H^{-1})$. Due to reflexivity of these spaces we obtain a subsequence, that we denote for notational simplicity also by $(u_m)_{m\in\N}$, such that
  \begin{equation}\label{eq:theo:existenzbeweis4}
    \begin{aligned}
    u_{m}\to u &\quad \text{weakly in $L^2([0,T], H^1_0(\Omega))$}, \\
    u^\prime_{m}\to w &\quad \text{weakly in $L^2([0,T], L^2(\Omega))$}, \text{ and } \\
    u^\pprime_{m}\to z &\quad \text{weakly in $L^2([0,T], H^{-1}(\Omega))$}.
  \end{aligned}
  \end{equation}
  It is easy to show that $w=u^\prime$ as well as $z=u^\pprime$.
  The argument showing that $u$ is in fact a weak solution to~\eqref{eq:arwp} and that it is unique can be found in~\cite{lionsmagenes:bvp1}.
  Finally, we take a look at the energy estimates for $u$. Because of the weak convergences in~\eqref{eq:theo:existenzbeweis4} we only obtain $L^2$-estimates in time at first, but it also follows that $(u_m)_{m\in\N}$ is bounded in  $L^\infty([0, T], H^1_0)$ and $(u_m^\prime)_{m\in\N}$ is bounded in $L^\infty([0, T], L^2)$.
  Up to extraction of a further subsequence (that we do again not denote explicitly), this shows that
  \begin{equation*}
    \begin{aligned}
      u_{m}\to u &\quad \text{weak-$\ast$\ in $L^\infty([0,T], H^1_0(\Omega))$}, \\
      u^\prime_{m}\to u^\prime &\quad \text{weak-$\ast$\ in $L^\infty([0,T], L^2(\Omega))$}, \text{ and } \\
      u^\pprime_{m}\to u^\pprime &\quad \text{weakly in $L^2([0,T], H^{-1}(\Omega))$}.
    \end{aligned}
  \end{equation*}
  The energy estimates of Lemma~\ref{lemma:energie} hence transfer to $u$ which in particular belongs to the space $H^2([0,T] , H^{-1}) \cap W^{1,\infty}([0,T], L^2) \cap L^\infty([0, T], H^1_0)$. \qedhere
\end{proof}

The last result states that for parameters that are first smooth enough and second close enough to a function bounded from below by some constant, there is a unique solution to the wave equation that satisfies an energy estimate.
Before stating this as a corollary, recall the number $\delta(\Omega)$ from~\eqref{eq:constantDelta} and the arbitrary, but fixed, constant $c_0 \in \R$ from Lemma~\ref{lemma:energie}.

\begin{corollary}\label{korollar:existenzeindeutigenergie}
  The wave equation~\eqref{eq:arwp} possesses for every $f\in L^2([0,T], L^2)$ and every
  \begin{equation}\label{eq:aux455}
  c\in  H^2([0,T], L^2) \ \cap \bigcup_{\substack{c_b \, \in L^\infty([0,T], L^2) \\[1mm] c_b \, \geq \, c_0 \, \ale}} B_{L^\infty([0,T], L^2)}(c_b, \, \delta(\Omega))
  \end{equation}
  a unique weak solution. This weak solution satisfies the energy estimate
  \begin{align*}
    \esssup_{t\in[0,T]}\norm{u(t)}_{H^1_0} + \esssup_{t\in[0,T]} & \norm{u^\prime(t)}_{L^2} +
    \norm{u^\pprime}_{L^2([0,T], H^{-1})}
    \leq e^{C \left(1+\norm{c}_{H^2([0,T], L^2)}\right)}\norm{f}_{L^2([0,T], L^2)}
  \end{align*}
  with $C$ depending only on $T$, $\Omega$ and $c_0$.
\end{corollary}

\section{Fréchet differentiability with respect to the parameter}
\label{sec:frechet}

The solution theory from the last section allows to define a solution operator $S: \, c\,\mapsto \, u$ mapping the time-dependent parameter $c$ to the wave $u$.
This operator is obviously non-linear since, e.g.~$c=0$ is not mapped to the trivial solution.
For inversion, we will hence exploit that $S$ can be locally linearized by its Fr\'echet derivative.

Differentiability of $S$ follows from Lipschitz continuity that we derive, by and large, via the energy estimate from Corollary~\ref{korollar:existenzeindeutigenergie}.
The precise setting is fixed in the following formal definition of $S$.
For simplicity, we \textit{fix} the source term $f\in L^2([0,T], L^2)$ for a moment and recall a last time the fixed constants $\delta(\Omega)$ from~\eqref{eq:constantDelta} and $c_0 \in \R$ from Lemma~\ref{lemma:energie}.

All constants $C$ we use in the sequel may change value from line to line but depend merely on $\Omega$, $c_0$, $T$, and an additional Lebesgue index $p$ that is fixed in the following definition.

\begin{definition}\label{def:operator}
  We consider $\map S{\mathcal D(S)\subset X}Y$ mapping the parameter $c$ to the solution $u$ of~\eqref{eq:theo:schwachelsg} as an operator with domain $\mathcal D(S)$ embedded in $X \assign H^2([0,T], L^2) \cap L^2([0,T], L^p)$ for a fixed $p=p(n)$ with $p(1)\geq 2$, $p(2)>2$, and $p(3)>3$,
  \[
    \mathcal D(S) \assign X \cap \bigcup_{\substack{c_b \, \in L^\infty([0,T], L^2) \\[1mm] c_b \, \geq \, c_0 \, \ale}} B_{L^\infty([0,T], L^2)}(c_b, \, \delta(\Omega)).
  \]
  Due to Corollary~\ref{korollar:existenzeindeutigenergie}, we further set the solution space $Y$ to
  \begin{equation}\label{eq:Y}
     Y \assign H^2([0,T] , H^{-1}) \cap W^{1,\infty}([0,T], L^2) \cap L^\infty([0, T], H^1_0).
  \end{equation}
  Both $X$ and $Y$ are equipped with their natural norms.
\end{definition}

Note that the last definition requires parameters $c$ to belong to $L^2([0,T], L^p)$ for some index $p$.
The reason behind is that the product $h(t) w(t)$ of an $L^p$- and an $H^1_0$-function on a bounded domain belongs to $L^2$, which becomes fundamental for Lemma~\ref{lemma:lipschitz}. This is due to the continuous embedding $H^1_0\hookrightarrow L^q$ for $1\leq q \leq \infty$ and $n=1$, $1\leq q < \infty$ and $n=2$, as well as $1\leq q < 6$ and $n=3$.
Precisely, for $h \in X$, $w \in Y$ and \ale~$t\in[0,T]$,
\begin{align}
  \| h(t) w(t) \|_{L^2} &\leq \begin{cases}
    \norm{h(t)}_{L^2} \norm{w(t)}_{L^\infty} &\leq C \norm{h(t)}_{L^2} \norm{w(t)}_{H^1_0}  \ \ \text{if $n=1$,} \\
    \norm{h(t)}_{L^p} \norm{w(t)}_{L^{\frac{2p}{p-2}}} &\leq C \norm{h(t)}_{L^p} \norm{w(t)}_{H^1_0} \ \ \text{if $n\in \{2,3\}$,}
 \end{cases} \label{eq:aux623}
\end{align}
which follows from the Hölder inequality, that is, $2p/(p-2) <\infty$ for $p>2$ and $2p/(p-2) <6$ for $p>3$.
Squaring and integrating~\eqref{eq:aux623} finally results in
\[
  \norm{hw}_{L^2([0,T], L^2)} \leq \norm{h}_{L^2([0,T], L^p)} \norm{w}_{L^\infty([0,T], H^1_0)} \leq \norm{h}_X \norm{w}_Y.
\]
Note further that the norm on $X$ is stronger than the $L^\infty([0,T], L^2)$-norm, such that $\mathcal D(S)$ is always open in $X$, which is crucial for Fr\'echet differentiability.

When treating inverse problems, we of course work with noisy data in $L^2([0,T], L^2)$; nevertheless, the following results profit from the somewhat more involved image space $Y$ from~\eqref{eq:Y}.

\begin{lemma}\label{lemma:lipschitz}
  The forward map $S$ is locally Lipschitz continuous.
\end{lemma}

\begin{proof}
  To $c_1$ and $c_2 \in \mathcal D(S)$ we assign $u_1\assign S c_1$ and $u_2\assign S c_2$. The difference $w\assign u_1- u_2$ then satisfies
  \begin{equation*}
    \dup{w^\pprime(t)}{\phi} + \int_{\Omega} \left[ \nabla w(t) \cdot \nabla \phi + c_1(t) w(t) \phi \right] \dx = \scp{(c_1(t)- c_2(t))u_2(t)}{\phi}
  \end{equation*}
  for \ale~$t\in [0,T]$ and all $\phi\in H^1_0$, subject to homogeneous initial values.
  Corollary~\ref{korollar:existenzeindeutigenergie} shows that
  \begin{align*}
    \big \|S c_1 &- S c_2 \big \|_Y= \norm{w}_Y  \leq \exp\left(C \left(1+\norm{c_1}_{H^2([0,T], L^2)}\right)\right) \norm{(c_1- c_2)  u_2}_{L^2([0,T], L^2)} \\
    &\leq C\exp\left(C \left(1+\norm{c_1}_{H^2([0,T], L^2)}\right)\right)\norm{c_1- c_2}_{L^2([0,T], L^p)} \norm{u_2}_{L^\infty([0,T], H^1_0)} \\
    &\leq C\exp\left(C \left(1+\norm{c_1}_{H^2([0,T], L^2)}+\norm{c_2}_{H^2([0,T], L^2)}\right)\right) \norm{f}_{L^2([0,T], L^2)} \norm{c_1- c_2}_X,
  \end{align*}
  where we exploited that $(c_1- c_2) u_2 \in L^2([0,T], L^2)$ can be estimated in norm as in~\eqref{eq:aux623} via the generalized H\"older's inequality, our choice of $p$, and the Poincar\'e estimate by the $X$-norm of $c_1 - c_2$ times  the norm of $u_2$ in $L^\infty([0,T], H^1_0)$.
  All arising constants can be uniformly bounded on every bounded set in $X$, which shows the claim.
\end{proof}

Formally computing the derivative of the weak formulation~\eqref{eq:theo:schwachelsg} of $u \in X$ with respect to $c \in \mathcal D(S)$ in direction $h\in X$ shows that $u_h \assign (S^\prime c)[h]$ needs to solve the variational formulation
\begin{equation*}
  \dup{u_h^\pprime(t)}{\phi} + \int_{\Omega} \left[ \nabla u_h(t) \cdot \nabla \phi + \left(h(t) u(t) + c(t) u_h(t)\right) \phi \right] \dx = 0
  \qquad \text{for all } \phi \in H^1_0
\end{equation*}
with zero initial conditions, i.e.~$u_h \in X$ is the weak solution to
\begin{equation}\label{eq:frechet:linpde}
u_h^\pprime - \laplace u_h + c u_h = - h u \quad \text{in $[0,T] \times\Omega$}
\end{equation}
with zero initial and boundary conditions.
The next theorem makes this formal argument rigorous.

\begin{theorem}
  The forward operator $S$ is Fr\'echet differentiable at $c\in \mathcal D(S)$:
  The derivative equals $S^\prime c \in \mathcal L(X,Y)$ and satisfies
  \[ S(c+h) - Sc = (S^\prime c)[h] + \mathcal O \left(\|h\|_X^2\right) \quad \text{as } h \to 0 \text{ in } X. \]
  For every $h\in X$ the function $u_h = (S^\prime c)[h] \in Y$ is the unique weak solution of the weak formulation
  \begin{equation}\label{eq:frechet:weakderiv}
    \dup{u_h^\pprime(t)}{\phi} + \scp{\nabla u_h(t)}{\nabla \phi} + \scp{c(t) u_h(t)}\phi = \scp{-h(t) u(t)}{\phi}
  \end{equation}
  for all $\phi\in H^1_0$, \ale~$t\in [0,T]$, and subject to homogeneous initial conditions $u_h(0)=u_h^\prime(0)=0$.
\end{theorem}

\begin{proof}
  As $\mathcal D(S)$ is open in $X$, there is for every $c \in \mathcal D(S)$ an open $X$-ball centered in $c$ such that all $h$ in this ball satisfy that $c+h\in \mathcal D(S)$.
  The solution $u_h$ from~\eqref{eq:frechet:weakderiv} is well-defined because the right-hand side belongs to $L^2([0,T], L^2)$; $u_h$ further satisfies the energy estimate from Corollary~\eqref{korollar:existenzeindeutigenergie} with $f$ replaced by $-h u$.
  The difference $w \assign S(c+h) - S c - (S^\prime c)[h]$ solves
  \begin{align*}
    \dup{w^\pprime(t)}{\phi} + \scp{\nabla w(t)}{\nabla \phi} + \scp{c(t) w(t)}\phi
    = \scp{h(t) \left(u(t)-u_+(t)\right)}{\phi}
  \end{align*}
  for all $\phi\in H^1_0$ and \ale~$t \in[0,T]$, with homogeneous initial conditions.
  Thus, Corollary~\ref{korollar:existenzeindeutigenergie} and~\eqref{eq:aux623} imply that
  \begin{align}
    \norm{u_+ - u - u_h}_Y &= \norm{w}_Y  \leq \exp\left(C \left( 1+ \norm{c}_X \right)\right)\norm{h (u-u_+)}_{L^2([0,T], L^2)} \label{eq:frechet:linerror}\\
    &\leq C\exp\left(C \left( 1+ \norm{c}_X \right)\right) \norm{h}_X \norm{u-u_+}_{L^\infty([0,T], H^1_0)}.\notag
  \end{align}
  Lipschitz continuity of $S$ now implies that
  \begin{align*}
    \norm{u_+ - u - u_h}_Y &\leq C\exp\left(C \left( 1+ \norm{c}_X + \norm{c+h}_X  \right)\right) \norm{f}_{L^2([0,T], L^2)}\norm{h}^2_X = \mathcal O\left(\norm{h}^2_X\right)
  \end{align*}
  as $h\to 0$ in $X$.
  Clearly, $S^\prime c$ is a linear operator, such that the energy estimate from Corollary~\ref{korollar:existenzeindeutigenergie} shows that $S^\prime c$ belongs to $\mathcal L(X, Y)$, which finishes the proof.
\end{proof}

The last proof shows that the definition of the parameter space $X$ in Definition~\ref{def:operator} as a subset of $L^2([0,T], L^p)$ is crucial to be able to bound the right-hand side $-hu$ in the variational formulation for the derivative $u_h$.

\section{Ill-posedness of the inverse problem}
\label{sec:illposed}

As we intend to reconstruct $c$ from wave measurements for several sources $f$ we still need to extend the forward operator $S$.
To this end, we first discuss suitable measurement operators $\psi$ to construct vector-valued solution and measurement operators.
For simplicity, we restrict ourselves to a particular linear and continuous measurement operator that we use later on for our numerical experiments.
Second, we rigorously define the inverse problem and prove its ill-posedness in several settings.

In applications, measurements are typically taken by sensors at fixed spatial positions $x_i^s$ and many instances of time $t_i^s$, which yields measured values of the wave at points $(t_1^s, x_1^s), \dots, (t_l^s, x_l^s) \in [0,T]\times \Omega$.
Unfortunately, our solution space $Y$ does not imply that point measurements of $u\in Y$ at $(t_i^s, x_i^s)$ depend continuously on $u \in Y$.
However, sensors anyway provide mean values of the wave in time and space over small regions around the introduced measurement points $(t_i^s, x_i^s)$.
Thus, we model measurements as a convolution in time and space of the wave field against an integral kernel $k: \, \R \times \R^n \to \R$ in $L^2(\R^{n+1})$,
\[
  u_k(t,x) \assign \int_0^T \int_\Omega k(t-s, x-y) u(s,y) \dy \ds, \quad t\in \R, \ x \in \R^n,
\]
and evaluate the convolved version $u_k$ of $u$ at the measurement points $(t_i^s, x_i^s)$.
Precisely, for numbers $r_t>0$ and $r_x>0$ we define a particular kernel $k$ via the normalized auxiliary function $g: \, \R \to \R$ that equals $\sqrt3\, (1-s)$ if $s<1$ and zero otherwise as follows,
\[
  k(t,x) = \frac{g(|t|/r_t) g(\norm{x}/r_x)}{r_t\,r_x^n}, \quad t\in \R, \ x \in \R^n.
\]
This kernel clearly belongs to $L^2(\R^{n+1}) \cap C(\R^{n+1})$ such that the convolution $u_k$ depends in the maximum norm continuously on $u \in L^2([0,T], L^2)$.
The associated evaluation operator
\begin{equation}\label{eq:psi}
 \map \psi{L^2([0,T], L^2)}{\R^l},\quad (\psi u)_i \assign u_k(t_i^s, x_i^s) = \int_0^T \int_\Omega k(t_i^s-s, x_i^s-y) u(s,y) \dy \ds.
\end{equation}
hence roughly speaking models sensor measurements close to $u(t_i^s, x_i^s)$ if $r_t$ and $r_x$ are small.
Due to the Cauchy-Schwarz inequality, $\norm{\psi}_{\mathcal L(L^2([0,T],L^2),\R^l)} \leq \sqrt{l} \,\norm{k}_{L^2(\R^{n+1})} = \sqrt{l}$.
Later on, we will require the adjoint $\psi^*$ of $\psi$ and hence note already here that $\psi^*: \, \R^l \to L^2([0,T], L^2)$ can be characterized by
\[
  \psi^* x = \sum_{i=1}^l x_i\, k(t_i^s-\cdot, x_i^s-\cdot), \qquad x \in \R^l.
\]
Finally, we allow for several sources $f_1, \dots, f_d \in L^2([0, T], L^2)$ for some $d\in\N$ as excitation for the wave equation~\eqref{eq:theo:schwachelsg}.
For any parameter $c$ that satisfies the assumptions of Corollary~\ref{korollar:existenzeindeutigenergie}, the source $f_i$ defines a wave $u_i\in Y$, such that for every $c\in X$ we can associate a forward operator $S_i$ to $f_i$ as in Definition~\ref{def:operator}.
The vector-valued solution operator $\mathbf S$ simply collects all $S_i$ in a vector (and we do \emph{not} distinguish column- and row vectors here).

\begin{definition}
  For the spaces $X$ and $\mathcal D(\mathbf S) \assign \mathcal D(S)$ as in Definition~\ref{def:operator} we set the vector-valued solution operator $\mathbf S$ as
  \[
    {\mathbf S}\colon{\mathcal D(\mathbf S) \subset X} \to{L^2([0, T], L^2)^d},
    \quad
     c \mapsto \left( S_1 c, \dots, S_d c \right).
  \]
  Further, the evaluation operator $\Psi: \, L^2([0,T], L^2)^d \to (\R^l)^d$ with $\Psi (u_1,\dots, u_d) = (\psi \, u_1, \dots, \psi \, u_d)$ is continuous, and the non-linear measurement operator mapping parameters to measurements for $d \in\N$ sources $f_1, \dots, f_d$ is
  \begin{equation}\label{eq:aux677} \map{\Phi}{X}{{(\R^l)}^d}, \quad \Phi = \Psi \circ \mathbf S. \end{equation}
\end{definition}

\autoref{abb:messung:abbildungen} illustrates connections between the vector-valued operators $\Phi$, $\Psi$, and $\mathbf S$. Of course, $\mathbf S$ and $\Phi$ are component-wise Fr\'echet differentiable.

\begin{figure}[tb]
  \centering
  \includegraphics{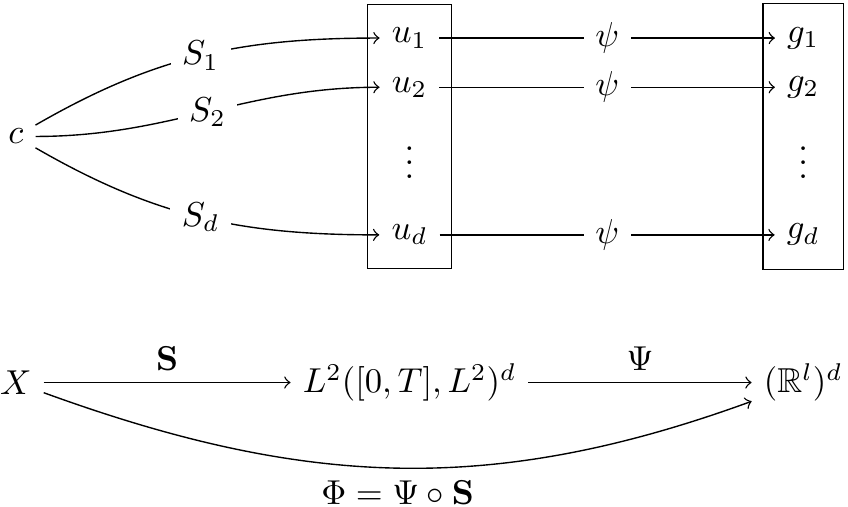}
  \caption{Vector-valued evaluation and measurement operators.}\label{abb:messung:abbildungen}
\end{figure}

The inverse problem we consider from now on is to determine the parameter $c \in X$ either from the solution $\mathbf Sc \in L^2([0,T], L^2)^d$ or from sensor measurements modeled by $\Phi c \in (\R^l)^d$.
We hence aim to determine the parameter of a differential equation from its impact on a wave and of course expect ill-posedness of this task.
This holds particularly if the data $\Phi c$ stems from sensors placed on a surface inside $\Omega$, since the unknown parameter then depends on more variables than the data.
Since $\Phi$ is by construction an operator with finite-dimensional range (and hence in particular features a closed range), the question of ill-posedness in the sense of the subsequent definition is anyway irrelevant for this measurement operator.
(Note the comment in the end of this section on different kinds of measurement models that partly touches this point.)
We prove in the rest of this section that the solution operator $S$ or its derivative $S'$ yield locally ill-posed inversion problems.
This directly implies ill-posedness of operator equations involving $\mathbf S$ or its linearization.

For a general operator $\map F{\mathcal D(F)\subset V}W$ between Banach spaces $V$ and $W$ we recall from~\cite[Definition~3.15]{schuster:regbanach} that the equation $Fx = y$ is locally ill-posed in $x^+ \in \mathcal D(F)$
if there exist for all $r>0$ sequences $(x_n)\subset B_r(x^+) \cap \mathcal{D}(F)$ such that $\| Fx_n  - F x^+\|_W \to 0$ but $\| x_n - x^+ \|_V \not \to 0$ as $n\to\infty$. (See~\cite[Definition 1.1]{HofmannScherzer98} for the corresponding definition in Hilbert spaces.)
As the notion of locality is meaningless for linear problems, a linear operator equation is either \emph{everywhere} locally ill-posed or else \emph{everywhere} locally well-posed.
For reflexive spaces $V$ and $W$ and $D(F)=V$, a linear operator equation is ill-posed if and only if the linear operator $F$ possesses a non-closed range or fails to be injective, see~\cite[Proposition~3.9]{schuster:regbanach}.

This last point gets essential when one linearizes $Sc =u$ via the Fr\'echet derivative $S'$ in $c^+ \in \mathcal D(S)$ and tackles $(S'c^+)[h] = g - Sc^+$ as an equation for $h\in X$. For the following results, recall that $Y=H^2([0,T], H^{-1})\cap W^{1,\infty}([0,T], L^2) \cap L^\infty([0,T], H^1_0)$ is the natural image space for $S$ that results from the energy estimates.

\begin{lemma}
  If $f\neq 0$, then $\map{S^\prime c}X{L^2([0,T], L^2)}$ is for all $c\in \mathcal D(S)$ a compact operator with infinite-dimensional range. In particular, $\mathrm{Rg}(S^\prime c)$ is not closed in $L^2([0,T], L^2)$ and the linearized operator equation $(S'c^+)[h] = g - Sc^+$ is locally ill-posed in every $h\in X$.
\end{lemma}

\begin{proof}
  We already know that $S^\prime c$ is bounded and linear from $X$ into $Y$.
  The embedding $Y\hookrightarrow L^2([0,T], H^1_0) \cap H^1([0,T], L^2)$ is continuous and from Simon~\cite[p.~85]{simon:compact} we know that the compact embedding of $H^1_0$ in $L^2$ implies that
  \[ L^2([0,T], H^1_0) \cap H^1([0,T], L^2) \hookrightarrow L^2([0,T], L^2) \]
  is compact as well.
  Thus, $\map{S^\prime c}X{L^2([0,T], L^2)}$ is compact and linear.
  If this operator possesses a finite-dimensional range, then the set of right-hand sides $-h \cdot Sc$ in the variational formulation~\eqref{eq:frechet:weakderiv} of $(S^\prime c)[h]$ must also belong to a finite-dimensional space by unique solvability of the wave propagation problem~\eqref{eq:theo:schwachelsg}.
  This forces $Sc$ and hence also $f$ to vanish, what we excluded in the lemma, and hence proves by contradiction that $\mathrm{Rg}(S^\prime c)$ cannot have finite dimension.
  As an infinite-dimensional range of a compact linear operator cannot be closed, we have shown the lemma's claim.
\end{proof}

The next lemma prepares a subsequent example on the ill-posedness of the linearized operator equation at $c=0$.

\begin{lemma}\label{illposed:lemmainjektiv}
  If $c\in \mathcal D(S)$ satisfies that $Sc\neq 0$ \ale~in $[0,T]\times \Omega$, then $S^\prime c$ is injective.
\end{lemma}

\begin{proof}
If $(S^\prime c)[h]$ vanishes for some non-zero $h\in X$, then the right-hand side $-h \cdot Sc$ of the formulation~\eqref{eq:frechet:weakderiv} for $(S^\prime c)[h]$ vanishes, which contradicts the assumption that $Sc\neq 0$ \ale.
\end{proof}

\begin{example}\label{bsp:illposed:1d}
  Set $n=1$, $\Omega = (0, \pi)$, $f(t,x) = 2 \sin(x)\sin(t)$, $c=0$, and fix $T>0$ arbitrarily.
  Then the range $\mathrm{Rg}(S^\prime c)$ is not closed in $Y$, i.e.~the linearized equation at $c=0$ is everywhere locally ill-posed between $X$ and $Y$ (and a fortiori also between $X$ and every Banach space containing $Y$).
\end{example}

\begin{proof}
  Separation of variables shows that the solution to $u^\pprime(t,x) - \partial_x^2 u(t,x) = 2 \sin(x) \sin(t)$ is given by $u(t,x) = \sin(x) \left ( \sin(t) - t\cos(t) \right)$.
  Obviously $u\neq 0$ almost everywhere, which means that $S^\prime c$ is injective.
  We define $(h_k)_{k\in \N} \subset X = H^2([0,T], L^2) \cap L^2([0,T], L^p)$ as $h_k(t,x)=k^{1/2}$ if $x<1/k$ and $0$ else.
  This sequence converges point-wise to zero, but not in the $X$-norm as
  $
    \norm{h_k}_X = \norm{h_k}_{L^2([0,T],L^2)} = T^{1/2}.
  $
  %
  The energy estimates imply that $u_h^k \assign (S^\prime c)[h_k]$ is bounded by
  \[ \norm{u_h^k}_Y \leq C \, \norm{u \, h_k}_{L^2([0,T], L^2)} \leq C (1+T) \sqrt{\pi T}\ \frac{\sin(1/k)}{\sqrt{1/k}} \to 0 \quad \text{as } k\to\infty.
  \]
  Due to injectivity of $S^\prime c$ we conclude that $S^\prime$ cannot have a bounded generalized inverse that is continuous, which proves the claim.
\end{proof}

The last example exploits the zeros of the solution $u(t,\cdot)$ at the boundary of $(0,\pi)$; analogous examples can be constructed independent of dimension, the right-hand side $f$, or the chosen boundary conditions, as long as $u$ is at least continuous and possesses zeros in $[0,T] \times \overline \Omega$.

In a Hilbert space setting there are many results known that connect the ill-posedness of a non-linear equation with the ill-posedness of its linearization, see, e.g.~Section 2 in~\cite{HofmannScherzer98}.
One example is the so-called tangential cone condition, see~\eqref{eq:illposed:nichtlin2} below or~\cite{Scherzer95}, which furthermore straightforwardly extends to Banach spaces.

\begin{theorem}\label{theorem:riederschlecht}
  Assume that $\map F{\mathcal D(F)\subset V}W$ is Fr\'echet differentiable between Banach spaces $V$ and $W$.
  If for some $x^+ \in \mathcal D(F)$ there are $r>0$ and $0 \leq \omega<1$ such that
  \begin{equation}
    \norm{F(v)-F(w) - F^\prime(w)[v-w]}_W \leq \omega \, \norm{F(v) - F(w)}_W \label{eq:illposed:nichtlin}
  \end{equation}
  holds for all $v,w\in B(x^+, r) \cap \mathcal D(F)$, then
  \begin{equation} \label{eq:illposed:nichtlin2}
    1-\omega \leq \frac{\norm{F^\prime(w)[v-w]}_W}{\norm{F(v) - F(w)}_W} \leq 1+\omega, \quad v\not = w.
  \end{equation}
  In this case, the non-linear problem $F(x) = y$ is locally ill-posed in $x^+$ if and only if the linearized problem at $x^+$ is locally ill-posed everywhere; that is, either $F^\prime(x^+)$ is not injective or $\mathrm{Rg}(F^\prime (x^+))$ is not closed in $W$.
\end{theorem}

\begin{proof}
  It is well-known that condition~\eqref{eq:illposed:nichtlin} transforms into~\eqref{eq:illposed:nichtlin2} by the reverse triangle inequality.
  If the latter condition holds, the linearized residual in the enumerator behaves as the non-linear residual in the denominator, such that the non-linear and the linearized operator equation can only be jointly locally ill-posed.
\end{proof}

Recall from Definition~\ref{def:operator} that we have fixed a dimension-dependent Lebesgue index $p \in [2,\infty)$ when setting up the parameter space $X$.
The next lemma exploits the conjugate index $q\in (2,\infty]$ defined by $1/p+1/q = 1/2$.

\begin{theorem}\label{lemma:illposed:anwendung}
  Define $q\in (2,\infty]$ as above.
  Then both settings $\map S{\mathcal D(S)\subset X}{L^\infty([0,T], H^1_0)}$ and $\map S{\mathcal D(S) \cap L^4([0,T], L^p)}{L^4([0,T], L^q)}$
  for the solution operator $S$ allow to prove the non-linearity condition~\eqref{eq:illposed:nichtlin}.
  As both $L^4([0,T], L^p)$ and $L^4([0,T], L^q)$ are reflexive, the conclusion of~\autoref{theorem:riederschlecht} holds for the second case, i.e.~local ill-posedness of the non-linear equation $Sc=u$ at some $c^+ \in \mathcal D(S)\cap L^4([0,T], L^p)$ is equivalent to the local ill-posedness of the corresponding linearized equation $(S'c^+)[h]=u-Sc^+$ between $L^4([0,T], L^p)$ and $L^4([0,T], L^q)$.
\end{theorem}

The following proof clearly shows that there are many more interesting settings for the pre-image and image space of $S$ than announced in the theorem.

\begin{proof}
  We have already estimated the linearization error in~\eqref{eq:frechet:linerror}:
  For $c_{1,2}\in \mathcal D(S)$,
  \[
    \norm{Sc_1 -Sc_2 - (S^\prime c_2)[c_1-c_2]}_Y \leq \exp\left(C\left(1+\norm{c_2}_X\right)\right) \norm{(c_1-c_2)\, (Sc_1 - Sc_2)}_{L^2([0,T], L^2)}.
  \]
  We estimate the last term on the right as in~\eqref{eq:aux623},
  \begin{align}
    \norm{(c_1-c_2)\, (Sc_1 - Sc_2)}_{L^2([0,T], L^2)}^2
    & \leq C \norm{ c_1-c_2}_{L^2([0,T], L^p)}^2 \norm{Sc_1 - Sc_2}_{L^\infty([0,T], L^q)}^2, \label{eq:aux852}
  \end{align}
  such that
  \begin{equation*}
    \norm{Sc_1 -Sc_2 - (S^\prime c_2)[c_1-c_2]}_Y \leq C \exp\left(C \left(1+\norm{c_2}_X\right)\right) \norm{ c_1-c_2}_{L^2([0,T], L^p)} \norm{Sc_1 - Sc_2}_{L^\infty([0,T], L^q)}  .
  \end{equation*}
  As the $X$-norm is stronger than the $L^\infty([0,T], L^q)$-norm,
  \[
    C\exp\left(C \left(1+\norm{c_2}_X\right)\right) \norm{c_1-c_2}_{L^2([0,T], L^p)} \leq \omega \quad \text{for some $\omega < 1$}
  \]
  holds whenever $c_1-c_2$ is small enough in the $X$-norm.
  Thus, the claimed bound~\eqref{eq:illposed:nichtlin} holds if we choose $X$ as pre-image space and $L^\infty([0,T], H^1_0)$ as image space for $S$, and $r>0$ small enough.
  (The $Y$-norm can always be bounded from below by the $L^\infty([0,T], H^1_0)$-norm.)

  To prove the analogous result for the pre-image space $X\cap L^4([0,T], L^p)$, one uses H\"older's inequality with twice the index $4$ instead of~\eqref{eq:aux852},
  \[
    \norm{(c_1-c_2)\, (Sc_1 - Sc_2)}_{L^2([0,T], L^2)}^2
    \leq C\| c_1-c_2 \|_{L^4([0,T], L^p)} \| Sc_1 - Sc_2 \|_{L^4([0,T], L^q)}.
  \]
  For the rest of the proof it is then sufficient to reduce $X$ to $X \cap L^4([0,T], L^p)$ and to set the image space to $L^4([0,T], L^q)$; notably, these choices merely yield reflexive Banach spaces.
  \qedhere
\end{proof}

Apart from condition~\eqref{eq:illposed:nichtlin2}, one can also show Lipschitz continuity of $c \mapsto S'c$ in the operator norm between $X$ and $L^2([0,T], L^2)$, that is, $\norm{S^\prime c - S^\prime c^+} \leq L\, \norm{c-c^+}_X$ holds for all $c\in B(c^+, r) \cap \mathcal D(S)$.
If one embeds the parameter space $X$ into a Hilbert space, this allows to show that local ill-posedness of $Sc = g$ at $c^+ \in \mathcal D(S)$ implies local ill-posedness of the linearized equation at $c^+$, see~\cite[Section 2]{HofmannScherzer98}.

Let us finally mention a classical result on the ill-posedness of equations involving $S$ composed with a measurement operator.
To this end, consider an operator $M$ such that the product $M \circ S$ is compact, continuous, and weakly sequentially closed  from, roughly speaking, all parameters in $H^2([0,T],L^2)$ that possess some lower real bound, into some separable Hilbert space $Z$.
(Precisely, the domain of definition is the set defined in~\eqref{eq:aux455} and included in $H^2([0,T],L^2)$.)
If $M \circ S$ additionally possesses an infinite-dimensional range, then the operator equation $(M\circ S) c = g$ is locally ill-posed at any parameter $c$ in this domain of definition, see, e.g.~\cite[Proposition A3]{Engl1989}.
This general result is independent of notions of derivatives and hence merely requires parameters $c$ in the set from~\eqref{eq:aux455}.

\section{Discretization of the wave equation}
\label{sec:discWave}

In this section we discuss the discretization of the wave equation~\eqref{eq:arwp} that we use to compute our numerical examples in Section~\eqref{sec:numerics}.
Recall that our existence theory treats the weak formulation of the wave equation,
\begin{equation*}
  \dup{u^\pprime(t)}{\phi} + \scp{\nabla u(t)}{\nabla \phi} + \scp{c(t)u(t)}{\phi} = \scp{f(t)}{\phi}
	\quad
	\text{for all $\phi\in H^1_0$ and \ale~$t\in[0,T]$,}
\end{equation*}
together with zero initial conditions $u(0) = u^\prime(0) = 0$.
We discretize the latter problem by Rothe's method, i.e.~we start by discretization in time, and consider the first-order system gained from $v\assign u^\prime$ as additional unknown,
\begin{equation}\label{eq:disk:system}
  \begin{aligned}
    \scp{u^\prime(t)}{\phi} - \scp{v(t)}{\phi} &= 0, \\
    \dup{v^\prime(t)}{\phi} + \scp{\nabla u(t)}{\nabla \phi} + \scp{c(t)u(t)}{\phi} &= \scp{f(t)}{\phi}.
  \end{aligned}
\end{equation}
Using a fixed step size $\Delta t > 0$ we obtain time steps
\[
  t_i \assign i\, \Delta t, \quad i=0,\dots, N-1, \quad N \assign 1 + \left\lceil T/\Delta t\right\rceil,
\]
write $u^i \assign u(t_i)$ and analogously $v^i$ and $f^i$ for $i=0, \dots, N-1$, and set
\[
a^i(\varphi_1,\varphi_2) = \scp{\nabla \varphi_1}{\nabla \varphi_2} + \scp{c(t_i) \varphi_1}{\varphi_2}, \quad \varphi_{1,2} \in H^1_0.
\]
We approximate all time derivatives in~\eqref{eq:disk:system} by a $\theta$-scheme, i.e.~a weighted average of forward- and backward difference quotients in $u^i$ and $v^i$. Further elmininating the dependence of the first equation on $v^i$ shows that $(u^i, v^i)$ solves
\begin{subequations}\label{eq:disk:semi}
  \begin{align}
    \scp{u^i}{\phi} + \theta^2 \Delta t^2 a^i(u^i, \phi) &= \scp{\theta \Delta t^2  \left(\theta f^i + (1-\theta) f^{i-1}\right) + u^{i-1} + \Delta t v^{i-1}}{\phi} \notag \\
    &\quad - \Delta t^2 \theta (1-\theta) a^{i-1}(u^{i-1}, \phi), \label{eq:disk:semi1}\\
    \scp{v^i}{\phi} &=  \scp{\Delta t \left(\theta f^i + (1-\theta) f^{i-1}\right) + v^{i-1}}{\phi} \notag \\
    &\quad - \Delta t\left (\theta a^i(u^i, \phi) + (1-\theta) a^{i-1}(u^{i-1}, \phi)\right),\label{eq:disk:semi2}
  \end{align}
\end{subequations}
for $i=1,\dots, N$, with initial values $u^0=v^0=0$.
Given $(u^{i-1}, v^{i-1})$, the first equation~\eqref{eq:disk:semi1} can be used to compute $u^i$ by solving one elliptic problem and then plug the result into~\eqref{eq:disk:semi2} to compute $v^i$ via a second elliptic problem.
We actually choose $\theta = 1/2$ to obtain the Crank-Nicolson scheme, which converges in each time step of second order as $\Delta t \to 0$.
The error of the last step and consequently the total error are hence of the first order in $\Delta t \to 0$. In addition, the scheme is unconditionally stable and does not exhibit energy loss, see~Larsson and Thomée~\cite{larsson:pde}.

To transform the semi-discrete system~\eqref{eq:disk:semi} into a fully discrete one we rely on the finite element method.
For technical simplicity we assume that $\Omega$ is a polygon and consider shape-regular and quasi-uniform triangulations $\mathcal S$ of $\Omega$ that satisfy $\overline \Omega = \bigcup_{S\in \mathcal S} \overline S$ and $\diam S \leq h$ for all $S\in \mathcal S$. For all simplexes $S\in \mathcal S$ we denote all affine mappings on $S$ as $\mathbb{P}_1(S)$ and introduce the finite-dimensional variational approximation spaces
\begin{equation*}
  V_h \assign \Set{\phi\in C(\overline\Omega) \cap H^1_0: \,  \left.\phi \right|_S \in \mathbb P_1(S) \text{ for all $S\in \mathcal S$}} \subset H^1_0,
	\quad
  h>0.
\end{equation*}
These finite-dimensional spaces define discrete approximations $u^i_h \in V_h$ and $v^i_h \in V_h$ to $u^i$ and $v^i$ by restricting the test function $\varphi$ in~\eqref{eq:disk:semi} to $V_h$, too.
Standard error estimates for, e.g.~the $L^2$-error between $u^i_h$ and $u^i$ indicate this error to be of second order in the diameter of the largest simplex of $\mathcal S$, see, e.g.~Brenner and Scott~\cite{brenner:fem}.

If we denote the nodal basis of $V_h$ by $\{ \phi_1, \dots, \phi_K \}$, then both $u^i_h$ and $v^i_h$ are represented by $K$ coefficients,
\begin{equation}
  u^i_h = \sum_{k=1}^K {\vec u}^i_k \phi_k, \quad v^i_h = \sum_{k=1}^K {\vec v}^i_k \phi_k. \label{eq:endlbasis}
\end{equation}
Linearity of~\eqref{eq:disk:semi} shows that the latter system is equivalent to the linear system of size $K\times K$ for $\vec u^i = \big({\vec u}^i_k \big)_k$ one gets by inserting~\eqref{eq:endlbasis} into~\eqref{eq:disk:semi}.
Let us define the mass matrix $M$ and stiffness matrix $A^i$ through
\begin{equation}
 M = \left(\scp{\phi_k}{\phi_j}\right)_{k,j=1,\dots, K} \in \R^{K\times K}, \quad A^i = \left(a^i(\phi_k,\phi_j)\right)_{k,j=1,\dots, K} \in \R^{K\times K} \label{eq:disk:matrizen}
\end{equation}
and abbreviate expressions involving $f$ as a vector $F^{i} = ( \scp{  \theta f^i + (1-\theta) f^{i-1}}{\phi_j} )_{j=1,\dots, K}$ in $\R^K$.
This establishes the fully discrete system
\begin{subequations}\label{eq:disk:voll}
\begin{align}
  \left(M+\theta^2 \Delta t^2 A^i\right) \vec u^i &= \theta \Delta t^2\, F^{i} + M \vec u^{i-1} + \Delta t M\,\vec v^{i-1} - \Delta t^2 \theta (1-\theta) A^{i-1}\, \vec u^{i-1}, \label{eq:disk:voll1}\\
  M \, \vec v^i &= \phantom{\theta}\Delta t \  \, F^{i} + M \vec v^{i-1} - \Delta t \theta A^i\, \vec u^i - \Delta t (1-\theta) A^{i-1}\, \vec u^{i-1},\label{eq:disk:voll2}
\end{align}
\end{subequations}
which is best solved using an iterative method like GMRES.\@
If $v^i$ is not needed for further computations then the solution of~\eqref{eq:disk:voll2} for $v^i$ may be omitted if $M v^i$ is stored instead. Since we regard $S$ as a map into $L^2([0,T], L^2)$ this is the case for us.

For the triangulation of $\Omega$, the bookkeeping of the basis functions, and the assembly of~\eqref{eq:disk:voll} we use the finite element toolbox \texttt{ALBERTA}~\cite{schmidt:alberta}.

Although $c$ is not the solution of a PDE we nevertheless discretize it as an element of $V_h$ at each of the time steps $t_i$, in the very same way as $u$. This approach has the advantage of not requiring additional data structures for searched-for parameters.

\section{Computation of adjoints of derivatives}
\label{sec:adjoints}

The discretization scheme~\eqref{eq:disk:voll} allows to numerically approximate $S$ and its derivative $S^\prime$.
The regularization scheme for $c$ that we present in Section~\ref{sec:illposed} however also requires an approximation of the (complex-valued) transpose operator $(S^\prime c)^*$.
For simplicity, we will rather require the adjoint operator later on since we artificially change into a Hilbert space framework in the next section. The importance of knowing such an operator is however already clear from linear regularization theory via filter functions.

From now on we consider $S^\prime$ to be a linear operator from $X$ into $L^2([0,T],L^2)$ and compute its transpose operator mapping $L^2([0,T],L^2)$ into $X^\prime$.
Note that $S^\prime$ from~\eqref{eq:frechet:weakderiv} can be decomposed as
\begin{equation}\label{eq:aux988}
  S^\prime c = L_c \circ M_c
\end{equation}
where, first, $\map {M_c}X{L^2([0,T], L^2)}$, $h \mapsto -uh$, multiplies $h$ by $-u = -Sc$.
Second, $L_c$ is a (weak) solution operator for the wave equation $w^\pprime - \laplace w + cw = g$ with variable right-hand side $g$,
\[
  \map {L_c}{L^2([0,T], L^2)}{L^2([0,T], L^2)}, \quad g \mapsto w,
\]
with zero initial and Dirichlet boundary conditions for $w$. By the above decomposition of $S^\prime c$ into two bounded linear operators we next compute the transpose $(S^\prime)^\ast$. The resulting numerical schemes will actually carry over to $\Phi^\prime c$ and $\mathbf S^\prime c$, such that we beforehand note the following corollary of~\eqref{eq:aux988}.
\begin{corollary}
Assume that $c\in \mathcal D(\mathbf S)$ and recall from~\eqref{eq:aux677} that $\Phi = \Psi \circ \mathbf S$.
  \begin{enumerate}[(1)]
  \item $\mathbf S$ and $\Phi$ are Fréchet-differentiable in $c$ and $\Phi^\prime c = \Psi \, \circ \, \mathbf S^\prime c$.
  For the solution operator $\mathbf L_c$ mapping $g\in L^2([0,T], L^2)^d$ to $(L_c\, g_1, \dots, L_c\, g_d)\in L^2([0,T], L^2)^d$
  and $\mathbf M_c \assign \left(M_{c,1}, \dots, M_{c,d}\right) \in \mathcal L(X, L^2([0, T], L^2)^d)$, with $M_{c, i} h \assign -h \, S_i c$ for $h\in X$, there holds
  \[
    \mathbf S^\prime c = \left( S_1^\prime c, \dots, S_d^\prime c \right) = \mathbf L_c \circ \mathbf M_c \in \mathcal L(X, L^2([0, T], L^2)^d).
  \]
  \item Further, $(\mathbf S^\prime c)^* = \mathbf M_c^* \circ \mathbf L_c^* \in \mathcal L(L^2([0, T], L^2)^d, X)$ and for $h\in L^2([0,T], L^2)^d$ there holds
  \[
    \mathbf L_c^*\, h = \left(L_c^*\, h_1, \dots, L_c^*\, h_d \right), \quad \mathbf M_c^*\, h = \sum_{i=1}^d M_{c,i}^* h_i \in X^\prime.
  \]
  \end{enumerate}
\end{corollary}

Let us now determine a numerically computable representation of the adjoint $L_c^\ast$ between $L^2([0,T], L^2)$ that is as usual characterized for $z$ and $f \in L^2([0,T], L^2)$ with $w\assign L_c f$ by
\begin{equation}\label{eq:theo:adjsol}
  \int_0^T \scp{z(t)}{w(t)} \dt = \scp{w}{z}_{L^2([0,T], L^2)} \overset != \scp{f}{L_c^* z}_{L^2([0,T], L^2)} = \int_0^T \scp{f(t)}{L_c^*z(t)} \dt .
\end{equation}
As $L_c^* z(t)$ belongs to $H^1_0$ we can replace the right-hand side by the weak formulation for $w$,
\begin{equation}\label{eq:frechet:lglatt}
  \int_0^T \scp{f(t)}{L_c^*z(t)} \dt =  \int_0^T \left[ \dup{w^\pprime(t)}{L_c^* z(t)} + \scp{\nabla w(t)}{\nabla L_c^* z(t)} + \scp{c(t) w(t)}{L_c^* z(t)} \right] \dt.
\end{equation}
Two partial integrations in time show by the initial conditions for $w$ that
\[
  \int_0^T \dup{w^\pprime(t)}{L_c^* z(t)} \dt = \int_0^T \dup{(L_c^* z)^\pprime(t)}{w(t)} \dt + w^\prime(T) (L_c^* z)(T) - w(T)(L_c^* z)^\prime(T).
\]
Hence,~\eqref{eq:theo:adjsol} is fulfilled if $h=L_c^* z \in Y$ is the weak solution to
\[
  h^\pprime(t) - \laplace h(t) + c(t) h(t) = z(t)
	\quad
	\text{for \ale} t\in [0,T],
\]
with zero Dirichlet boundary values and zero \emph{end conditions} $h(T) = h^\prime(T)=0$.
Due to the theory in~Section~\ref{sec:solutiontheory} the latter differential equation is uniquely solvable and $z\mapsto h$ defines a bounded linear operator on $L^2([0,T], L^2)$ that is numerically evaluated in the same way as $S$.

We now turn to the transpose $M_c^\ast: \, L^2([0,T], L^2) \to X^\prime$ of the multiplication operator $M_c$.
(The dual space $X^\prime$ is computed for a weighted inner product of $H^2([0,T]; L^2)$, see below.)
Writing $M_c^* z = w_z \in X^\prime$ for $z\in L^2([0,T], L^2)$, the function $w_z \in X$ has to satisfy the equality
\begin{equation}\label{eq:disk:mstar}
  \scp{-uh}{z}_{L^2([0,T], L^2)} = \dup{h}{w_z}_{X\times X^\prime}
	\quad
	\text{for every $h\in X$.}
\end{equation}
Loosely speaking, $w_z$ can hence be interpreted as a smoothed version of $-uz$.
As $M_c^*$ is the only operator in our reconstruction scheme mapping into $X$, it actually controls smoothing of the searched-for parameter.
It is practical to steer this smoothing by weights $\alpha, \beta > 0$ that define the following duality product, extending the analogous weighted inner product of $L^2([0,T], L^2)$,
\begin{equation}\label{eq:dualproduct}
  \dup{g}{f}_{X\times X^\prime} \assign \dup{g}{f}_{L^2([0,T], L^p) \times L^2([0,T], L^{q})} + \alpha \scp{g^\prime}{f^\prime}_{L^2([0,T], L^2)} + \beta \scp{g^\pprime}{f^\pprime}_{L^2([0,T], L^2)}
\end{equation}
for all $g\in X$, the dual Lebesgue index $q \in [1,2]$ such that $1/p+1/{q} = 1$, and $f\in X^\prime = H^2([0,T], L^2) \cap L^2([0,T], L^{q})$.
Thus, we obtain from~\eqref{eq:disk:mstar} that
\begin{align*} 
  \int\limits_0^T &\dup{h(t)}{-u(t)z(t)}_{L^p\times L^{q}} \dt
  = \int\limits_0^T \left[ \dup{h(t)}{w_z(t)}_{L^p\times L^{q}} + \alpha \scp{h^\prime(t)}{w_z^\prime(t)} + \beta \scp{h^\pprime(t)}{w_z^\pprime(t)} \right] \dt,
\end{align*}
which is a weak formulation of a fourth-order differential equation in time.
Discretization of the last problem for $w_z$ via finite elements seems most natural but is indeed tedious as conforming finite element spaces need to be $H^2$-smooth in time, which is typically not pre-coded in open finite element packages.

Following the latter idea by via finite-dimensional subspaces of $X^\prime$ leads into Banach-space valued regularization schemes that we do merely for simplicity not consider in this paper.
Instead, we formally use a simpler finite difference scheme in each individual spatial degree of freedom that arises by first discretizing the latter problem in space:
As in the last section, we represent test functions $h\in X$ as $h(t,x) = \sum_{k=1}^K \vec h(t)_k \phi_k(x)$ such that $h(t)\in V_h$ for every $t\in [0,T]$ with some $\vec h = (\vec h_k)_{k=1,\dots, K} \in H^2([0,T])^K$.
In the same way we define $\vec u, \vec z \in L^2([0,T])^K$, $\vec w_z\in H^2([0,T])^K$,
and recall the mass matrix $M$ from~\eqref{eq:disk:matrizen}.
We denote by $\vec u \lcdot \vec z$ the component-wise multiplication of $\vec u$ and $\vec z$, and deduce by (formal) partial integration that
\begin{align*}
 - \int_0^T &\vec h(t)^\top  M  \left(\vec u(t) \lcdot \vec z(t)\right) \dt = \int_0^T \left[\vec h(t)^\top  M  \vec w_z(t) + \alpha  \vec h^\prime(t)^\top  M  \vec w_z^\prime(t) + \beta   \vec h^\pprime(t)^\top  M  \vec w^\pprime_z(t) \right] \dt \\
 &= \int_0^T \vec h(t)^\top  \left(M  \vec w_z(t) - \alpha  M  \vec w_z^{\prime\pprime}(t) + \beta   M  \vec w^{(4)}_z(t) \right) \dt \\
 &\qquad\quad+ \left[ \alpha  \vec h(t)^\top  M  \vec w_z^\prime(t) + \beta  \vec h^(t)^\top  M  \vec w_z^\pprime(t) - \beta  \vec h(t)^\top  M  \vec w_z^{\prime\pprime}(t) \right]_0^T.
\end{align*}
The above equation is fulfilled if $v_k \assign (M \vec w_z)_k$ solves for $k=1,\dots, K$ the one-dimensional ordinary differential equations
\begin{equation}\label{eq:aux1102}
  \begin{cases} v_k - \alpha  v_k^\pprime + \beta v_k^{(4)} = - \left( M \left(\vec u \lcdot \vec z\right) \right)_k, &\\
  v_k^\pprime(0) = v_k^\pprime(T) = 0, \quad \alpha v_k^\prime(0) = \beta v_k^{\prime\pprime}(0), \quad \alpha v_k^\prime(T) = \beta v_k^{\prime\pprime}(T). &
  \end{cases}
\end{equation}
In our numerical examples, we solve these systems at the time points $t_i = i\, \Delta t$, $i=0,\dots, N-1$, that we already fixed when solving for $u \in Y$ or representing $c \in X$. We continue to use the notation $v_k^i \assign v_k(t_i)$ and replace the appearing time derivatives by the standard central difference quotients up to order four, see, e.g.~\cite{fornberg:finitedifference}.
%
%
The resulting fully discrete equations at the time points $t_i$ then read
\begin{align}
  - \Delta t^4 ( M (\vec u^i &\lcdot \vec z^i) )_k = \Delta t^4 v_k^i \! - \alpha \Delta t^2 \left(v_k^{i+1} - 2 v_k^i + v_k^{i-1}\right)\! +\! \beta \left( v_k^{i+2}\! - 4 v_k^{i+1}\! + 6 v_k^i - 4 v_k^{i-1} + v_k^{i-2} \right) \notag \\
  &\hspace{-0.75cm}= \beta v_k^{i-2} - \left(\alpha \Delta t^2 + 4\beta \right) v_k^{i-1} + \left( \Delta t^4 + 2 \alpha \Delta t^2 + 6 \beta \right) v_k^i - \left(\alpha \Delta t^2 + 4\beta \right) v_k^{i+1} + \beta v_k^{i+2}.\label{eq:adjm1}
\end{align}
The latter equation requires \enquote{imaginary} nodes at $t_{-2}, t_{-1}, t_N$ and $t_{N+1}$ that enforce the boundary conditions at $i=0$ (for $t=0$) and $i=N$ (for $t=T$) in the second line of~\eqref{eq:aux1102}, i.e.
\begin{align}
  0 &= v_k^{-1} - 2 v_k^0 + v_k^1, \\
  0 
    &= \beta v_k^{-2} - (2\beta + \alpha \Delta t^2) v_k^{-1} + (2\beta + \alpha \Delta t^2) v_k^1 - \beta v_k^2, \\
  0 &= v_k^{N-2} - 2 v_k^{N-1} + v_k^N, \\
  0 &= \beta v_k^{N-3} - (2\beta + \alpha \Delta t^2) v_k^{N-2} + (2\beta + \alpha \Delta t^2) v_k^N - \beta v_k^{N+1}.\label{eq:adjm2}
\end{align}
To sum up, our scheme for the numerical evaluation of $M_c^* z$ works as follows:
\begin{enumerate}[1.]
  \item Compute the matrix-vector products $M (\vec u^i\lcdot \vec z^i)$ for $i=0, \dots, N-1$.
  \item Solve the $N+4$-dimensional linear system~\eqref{eq:adjm1}-\eqref{eq:adjm2} for every degree of freedom $k=1, \dots, K$ in space and as a result obtain the vector $(v_k^i)_{i=0,\dots, N-1}$.
  \item Determine the coefficients of $\vec w_z(t_i)$ of $(M_c^* z)(t_i)$ with respect to the basis of $V_h$ as the solution of $M \vec w_z(t_i) = (v^i_k)_{k=1, \dots, K}$ for every time step $t_i$.
\end{enumerate}
In our numerical examples in Section~\ref{sec:numerics} the number $K$ of degrees of freedom in space will be much higher than the number $N$ of time steps. (Typical orders of magnitudes for $n=3$ are $N\approx 200$ and $K\approx 5000$.)
Thus, the execution of step 1 in the scheme above as well as the solution of $N$ linear systems with the sparse mass matrix in step 3 can be done fast compared to, e.g.~the numerical solution of a forward problem. Since the matrix of the $(N+4)\times (N+4)$ dimensional system~\eqref{eq:adjm1}-\eqref{eq:adjm2} does not change throughout the reconstruction it can be factored once and then used for the fast solution of the $K$ equations in step 2. We use the library \texttt{SuperLU}~\cite{li:superlu} for this, which exploits sparsity of the linear system.

Note that the computation of $M_c^*$ would greatly simplify if $c$ did not have to be so smooth in time.
If, e.g.~the solution operator $S$ turned out to be well-defined in some open subset of $H^1([0,T], L^2)$, then the simpler equation $v_k - \alpha \, v_k^\pprime = - (\, M \left(\vec u \lcdot \vec z\right) )_k$ would arise in~\eqref{eq:aux1102} for $\alpha>0$, subject to homogeneous end conditions.
This problem could be easily approximated with finite elements.
If the entire setting even required no smoothness of $c$ at all, then the variant of the multiplication operator $M_c$ operating on $L^2([0, T], L^2)$ would even become self-adjoint.

\section{Regularization using inexact Newton iterations}
\label{sec:reginn}

In the preceding sections we showed well-definedness, continuity and differentiability of the solution operator $S$ and the measurement operator $\Phi$.
We assume now that
\begin{subequations}
  \begin{align}
    \Phi c^+ &= g^+ \quad \text{for }c^+\in \mathcal D(S) \subset X,\ g^+ \in (\R^l)^d, \label{eq:reg:phicg}\\
    \mathbf S c^+ &= u^+ \quad \text{for } c^+\in \mathcal D(S) \subset X,\ u^+ \in L^2([0,T], L^2)^d, \label{eq:reg:scu}
  \end{align}
\end{subequations}
to take a look at a particular regularization scheme that stably approximate $c^+$ from data.
More precisely, we propose inversion by the REGINN (\enquote{REGularization based on INexact Newton iteration}) algorithm, which was stated and analyzed by Rieder~\cite{rieder:reginn}.
We give a brief reminder how REGINN works, by considering merely the first inverse problem in~\eqref{eq:reg:phicg}.

In the entire section we actually neglect that the pre-image space $X$ is a Banach- instead of a Hilbert space. In our numerical  experiments, we instead use $H^2([0,T], L^2) \supset X$ as Hilbert space, as more involved schemes in Banach spaces are out of the scope of this paper.

Of course, we do not assume that data $g^+$ can be measured exactly but instead suppose to know some noisy version $g^\epsilon$ with relative noise level $\epsilon>0$, i.e.~$\norm{g^+-g^\epsilon} \leq \epsilon \norm{g^+} \approx \epsilon \norm{g^\epsilon}$.
As is customary, we assume to know $\epsilon>0$ a-priori.
We already mentioned that REGINN relies on successive linearization of~\eqref{eq:reg:phicg} starting with some initial guess $c_0\in \mathcal D(S)$ to generate a sequence $(c_k)_{k\in\N_0}\subset \mathcal D(S)$ of approximations of $c^+$.
Writing $c^+ = c_k + s_k^+$ for each $k\in \N_0$, the best update $s_k^+$ solves
\[
  (\Phi^\prime c_k)[s_k^+] = g^+ - \Phi c_k - E(c^+, c_k) \assigns b_k^+.
\]
Because of the linearization error $E(c^+, c_k)$ and the exact data $g^+$ we only know a perturbed right-hand side $b_k^\delta = g^\epsilon - \Phi c_k$. Its noise level $\delta$ is also unknown, as we only have $\|b_k^\delta - b_k^+\| \leq \epsilon \norm{g^\epsilon} + \mathcal O(\|c^+-c_k\|^2)$. REGINN applies a regularization method for linear inverse problems to this problem and stops it when the relative linear residuum is smaller than a tolerance times the non-linear residuum.
In our case the former is done via the method of conjugate gradients (CG), which creates an inner iteration that computes a sequence of approximations $(s_{k,i})_{i\in\N}$ of $s_k^+$. The stopping is done by choosing tolerances $\mu_k\in (0,1)$ and picking $s_k\assign s_{k,i_k}$ with
\begin{equation}\label{eq:reg:reginnabbruch}
  i_k \assign \min \Set{ i\in \N | \|(\Phi^\prime c_k) s_{k,i} - b_k^\delta \| < \mu_k \|b_k^\delta \|}.
\end{equation}
Afterwards we can set $c_{k+1} \assign c_k + s_k$ and continue the iteration,
which we stop using the discrepancy principle by a fixed parameter $\tau>1$,
\begin{equation}\label{eq:reg:reginndisc}
  k^* = k^*(\epsilon, g^\epsilon) \assign \min \Set{k\in \N | \|\Phi c_k - g^\epsilon \| \leq \tau \epsilon \, \|g^\epsilon \|}.
\end{equation}
The combination of REGINN with CG as inner regularization method was also analyzed by Rieder~\cite{rieder:cgreginn}.
Convergence is only guaranteed if the $\mu_k$ stay in the interval $[a,b]\subset (0,1)$ where $a$ and $b$ depend on unknown constants, like $\eta$ in the non-linearity condition.
Due to the shrinking linearization error we want to be able to reduce $\mu_k$ during the outer iteration.
On the other hand this reduction should not increase the computing time (number of CG-steps) of the next outer step too much.
Rieder proposes the following strategy in~\cite{rieder:reginn}:
Start with $\mu_1 = \mu_2 = \mu_{\text{start}}\in (0,1)$ and for $k\geq 3$ define
\begin{equation}\label{eq:reginn:muk}
  \tilde \mu_k =
  \begin{cases}
    1 - \frac{i_{k-2}}{i_{k-1}} ( 1 - \mu_{k-1}) & \text{if } i_{k-1} > i_{k-2}, \\
    \gamma \,\mu_{k-1} & \text{else.}
  \end{cases}
\end{equation}
The tolerance $\mu_k$ is then set to
\begin{equation} \label{eq:aux1285}
  \mu_k = \mu_{\text{max}} \max\left \{\tau \epsilon {\big \|g^\epsilon\big \|}/{\big \| g^\epsilon - \Phi c_k \big \|}, \tilde \mu_k\right \}
\end{equation}
where $\mu_{\text{max}} \in (\mu_\text{start}, 1)$. This achieves a $\gamma$-linear reduction of $\mu_k$ with $\gamma\in(0,1)$ if the number of inner steps is decreasing. We use $\mu_{\text{start}} = 0.7$, $\gamma = 0.9$ and $\mu_{\text{max}} = 0.99$.
Algorithm~\ref{alg:reginn} lists a pseudo-code for the whole reconstruction procedure.

\begin{algorithm}
\caption{REGINN for solving $\Phi c = g$}
\label{alg:reginn}
\begin{algorithmic}
  \State\textbf{Required}:\ Hilbert Spaces $X, Y$, $\map \Phi X Y$,
  \State\phantom{\textbf{Required}:\ }\ $c_0\in X$, $g^\epsilon \in Y$, $\norm{g-g^\epsilon} \leq \epsilon \norm{g^\epsilon}$
  \State$k\gets 0$
  \While{$\norm{\Phi c_k - g^\epsilon} > \epsilon \norm{g^\epsilon}$}
    \State$k\gets k+1$
    \State$\mu_k \gets$ parameter adaption rule~\eqref{eq:aux1285}
    \State$s_k\gets 0$
    \While{$\norm{(\Phi^\prime c_k)[s_k] - g^\epsilon - \Phi c_k} > \mu_k \norm{g^\epsilon - \Phi c_k}$}
      \State$s_k \gets$ next CG-iterate for equation $(\Phi^\prime c_k)[s] = g^\epsilon - \Phi c_k$
    \EndWhile%
    \State$c_k \gets c_{k-1} + s_k$
  \EndWhile%
\end{algorithmic}
\end{algorithm}

\section{Numerical examples}
\label{sec:numerics}

We want to show in this last section that REGINN with the CG-iteration is indeed able to provide an estimate of a time- and space-dependent parameter $c$ in acceptable time, especially for $n=3$.
To this end, we set $T\assign 2$, $\mathds 1 \assign (1)_{i=1,\dots, n}$, and reconstruct two different parameters from artificial data measured in $\Omega \assign (0,1)^n$ with $n\in \{1,2,3\}$.
The first parameter is hat-shaped and moves in time from $1/4 \cdot \mathds 1$ to $3/4\cdot \mathds 1$,
\[
  c_{\text{hat}}(t,x) \assign 20 \, h \!\left(4\, \norm{x - \tfrac{1+t}4 \mathds 1} \right),
	\quad
	\text{with }
  h(s) \assign \begin{cases}
  \exp\left(1-\frac{1}{1-s^2}\right)  \quad &\text{if $|s|<1$,} \\
  0 & \text{else.}
  \end{cases}
\]
This parameter is smooth in time and space, at least in theory.
Since spatial smoothness is actually not required by the parameter space $X$, we also test a parameter with discontinuities,
\[
  c_{\text{plateau}}(t,x) \assign \begin{cases}
  20 (1-|t-1|^2) \quad &\text{if $\norm{x-\tfrac12 \mathds 1}<\tfrac 14$,} \\
  0 & \text{else.}
  \end{cases}
\]
In all calculations we use the finite element interpolation of $c_{\text{plateau}}$ in $V_h$, which is of course continuous but has a sharp edge at $\partial B(\tfrac 12 \mathds 1, \tfrac 14)$.

Motivated by possible applications we set $d \assign 2^n$ at positions into the domain $\Omega$; each of the elements of their position vectors $x^a_k$ for $k\in \{1,\dots, d\}$ either equals to $1/3$ or 2/3.
We further consider that each actuators excites a wave in $\Omega$ that we model by $d$ right-hand sides $f_1,\dots, f_d$.
Precisely, for frequency $\omega = 8\pi$ and actuator radius $r^a = 0.1$ we define $\map {f_k}{\R \times \R^3}\R$ by
\[
	f_k(t,x) \assign \left( 1 - \frac{\norm{x-x^a_k}}{r^a} \right) \sin(\omega t) \quad \text{if}\ \norm{x-x^a_k}\leq r^a \text{ and } t \geq 0,
\]
and $f_k(t,x) = 0$ else.

For the discretization we define $\Delta t\assign 10^{-2}$ and employ a spatial grid consisting of $6$, $5$ or $4$ global refinements of the trivial triangulation of $(0,1)^n$ for $n=1,2$ or $3$, respectively.
The finite element interpolations of both parameters evaluated at $t=1.7$ are shown in \autoref{abb:vergleich:params}.
(Here and in subsequent figures we restrict ourselves to $n=2$.)

\begin{figure}[t]%
  \centering
  \subfloat[$c_{\text{hat}}(1.7, \cdot)$.]{\label{abb:vergleich:cmovinghat}
  \includegraphics{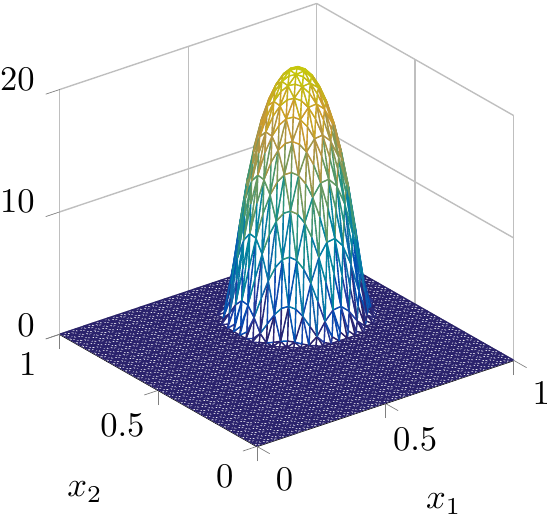}}%
  \subfloat[$c_{\text{plateau}}(1.7, \cdot)$.]{\label{abb:vergleich:cplateau}
  \includegraphics{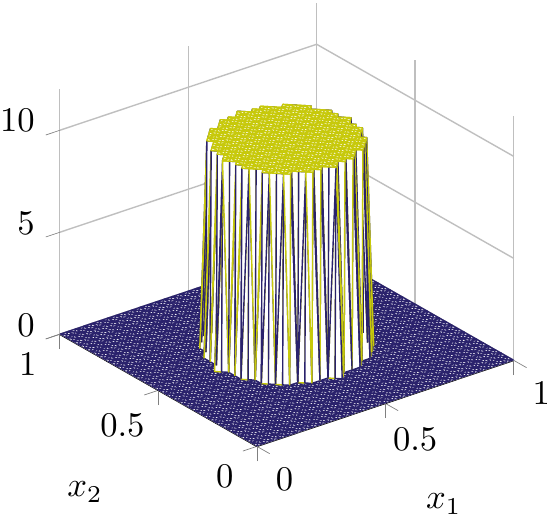}}%
  \caption{Finite element interpolations of both parameters in the case $n=2$.}\label{abb:vergleich:params}
\end{figure}%

In~\eqref{eq:dualproduct} we introduced numbers $\alpha$ and $\beta$ in front of the first and second order terms of the $H^2([0,T], L^2)$~scalar product in order to control the smoothness of the reconstruction.
Our primary goal is to minimize the $L^2([0,T], L^2)$-error to the exact parameter. For this we chose $\alpha$, $\beta$ in such a way that numerical approximations of $\alpha \|c^\prime \|^2$ and $\beta \|c^\pprime \|^2$ are one order of magnitude smaller than $\|c\|^2$. Tests with both parameters led us to define $\alpha \assign 2 \cdot 10^{-2}$ and $\beta \assign 2 \cdot 10^{-3}$.

We start our numerical experiments by checking whether the reconstruction $c_{k^*(\epsilon)}$ converges to the exact parameter in the $X$- or the $L^2([0,T], L^2)$-norm when $\epsilon$ tends to $0$.
We do so by applying REGINN to artificial data $u^\epsilon\in L^2([0,T], L^2)^d$ with relative noise level $\epsilon>0$, i.e.~${\|u^\epsilon - u\|}_{L^2([0,T], L^2)^d} = \epsilon \norm{u}$ for $u=\mathbf Sc$. The additive noise consists of a scaled vector of uniformly distributed pseudo-random numbers in $[-1,1]$.
The stopping index $k^*(\epsilon)$ is determined by the discrepancy principle with $\tau=2$, see~\eqref{eq:reg:reginndisc}.

\begin{figure}[!htb]%
  \centering
  \subfloat[$c_{k^*}(1.0, \cdot)$ for $c_{\text{hat}}$.]{
  \includegraphics{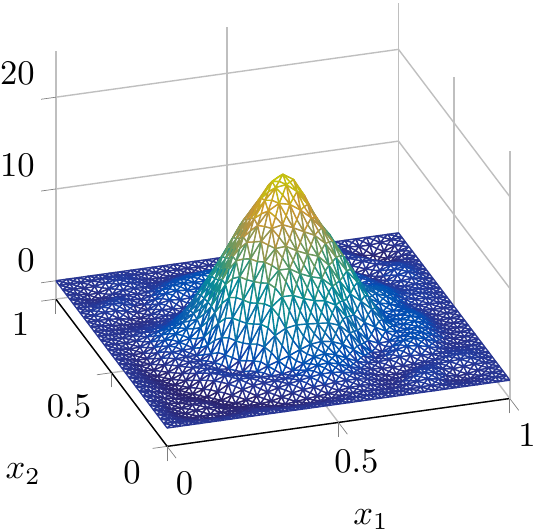}
  }
  \subfloat[$c_{k^*}(1.7, \cdot)$ for $c_{\text{hat}}$.]{
  \includegraphics{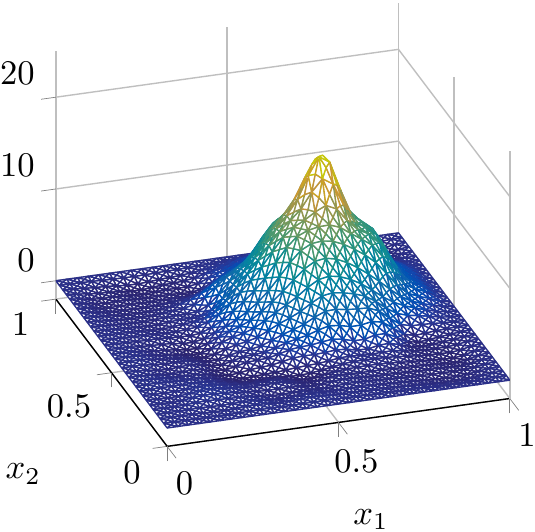}
  }

  \subfloat[$c_{k^*}(1.0, \cdot)$ for $c_{\text{plateau}}$.]{
  \includegraphics{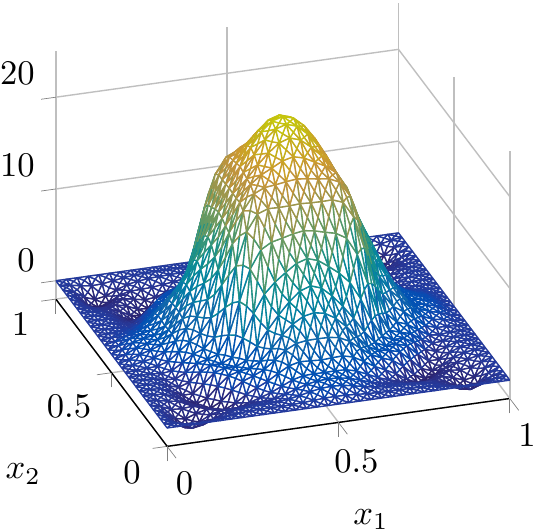}
  }
  \subfloat[$c_{k^*}(1.7, \cdot)$ for $c_{\text{plateau}}$.]{
  \includegraphics{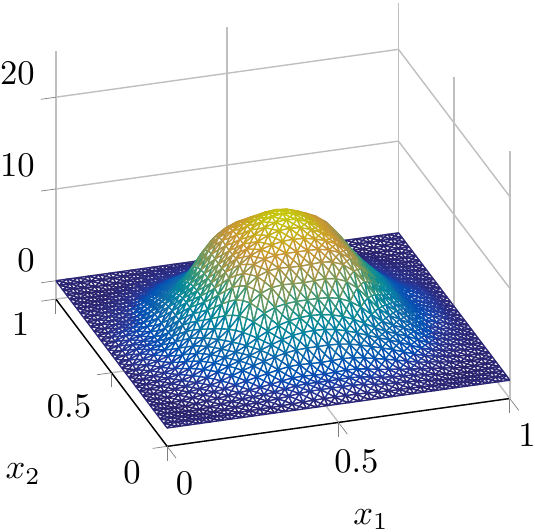}
  }
  \caption{Reconstructions from $u^\epsilon$ with $\epsilon=10^{-2}$ for $n=2$.}\label{abb:vergleich:paramsrekonstru}
\end{figure}

For $\epsilon=10^{-2}$ both reconstructions are satisfactory, as can be seen in \autoref{abb:vergleich:paramsrekonstru}.
Although the time dependence can already be deduced from these reconstructions, the $L^2$-errors are relatively high and amount to $45\%$ for $c_{\text{hat}}$ and $37\%$ when estimating $c_{\text{plateau}}$. The corresponding values in one and three spatial dimensions as well as the $H^2$-errors are listed in \autoref{tab:vergleich:paramsrekonstru}. In all dimensions the $H^2$-error for the moving hat is much higher than the $L^2$-error; for the other parameter both norms yield similar values.

\begin{table}[htbp]
  \centering
  \begin{tabularx}{0.87\textwidth}{Xcccc}\toprule
     & \multicolumn{2}{c}{$c_{\text{hat}}$} & \multicolumn{2}{c}{$c_{\text{plateau}}$} \\
     & $L^2([0,T], L^2)$ & $H^2([0,T], L^2)$ & $L^2([0,T], L^2)$ & $H^2([0,T], L^2)$ \\ \midrule
    $n=1$ & $26.28 \%$ & $57.93 \%$ & $29.47 \%$ & $33.94 \%$ \\
    $n=2$ & $44.68 \%$ & $71.90 \%$ & $37.06 \%$ & $41.15 \%$ \\
    $n=3$ & $58.85 \%$ & $81.66 \%$ & $40.71 \%$ & $43.50 \%$ \\
    \bottomrule
  \end{tabularx}
  \caption{Errors of the reconstruction from $u^\epsilon$ for $\epsilon=10^{-2}$.}\label{tab:vergleich:paramsrekonstru} 
\end{table}

\begin{figure}[!htbp]%
  \centering
  \subfloat[$c_{\text{hat}}$.]{
   \includegraphics{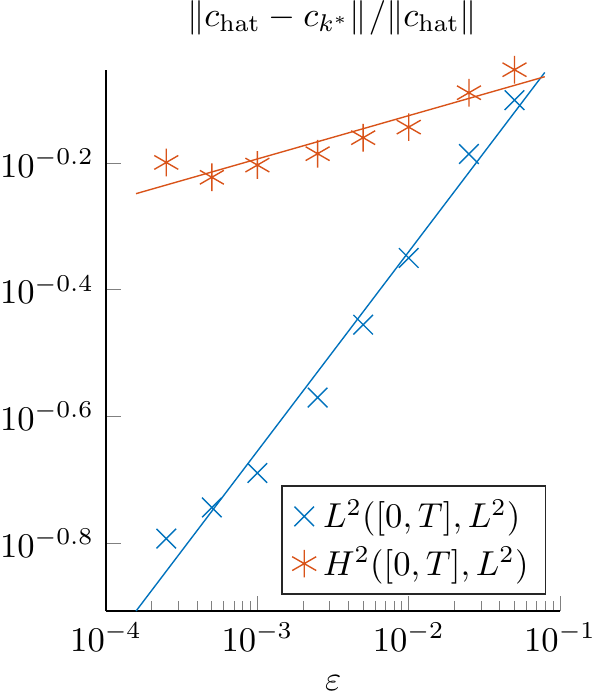}
  }
  \subfloat[$c_{\text{plateau}}$.]{
   \includegraphics{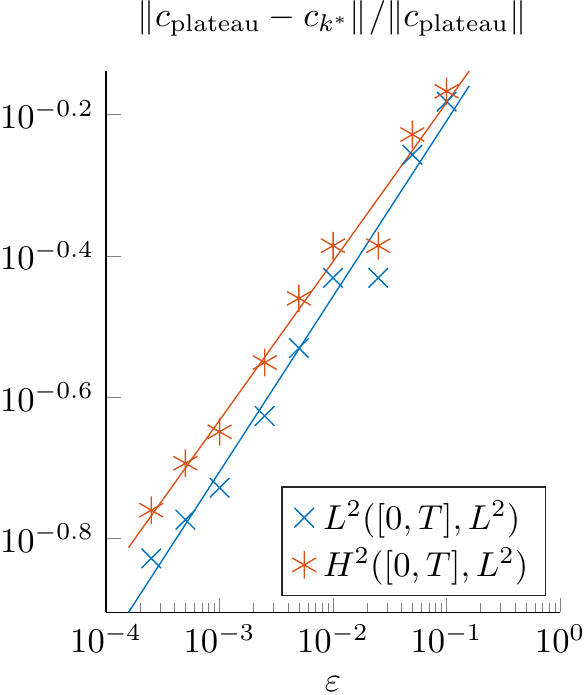}
  }
  \caption{Dependence of the reconstruction error on $\epsilon$ when reconstructing from $u^\epsilon$ in the case $n=2$.}\label{abb:vergleich:konvergenzcg}
\end{figure}%

\autoref{abb:vergleich:konvergenzcg} shows the dependence of these errors on $\epsilon$ in a logarithmic scale.
While both errors clearly converge for $c_{\text{plateau}}$, this is at least questionable for the $H^2$-error of the moving hat.
This leads to the hypothesis that $c_{\text{hat}}$ is not sufficiently smooth in time, which also seems to be the case in one and three space dimensions, see~\autoref{tab:vergleich:konvergenz}.
The $L^2$-error is approximately of order $\mathcal O(\epsilon^{0.3})$, at least for $\epsilon \in [2.5\cdot 10^{-4}, 5 \cdot 10^{-2}]$. For very small $\epsilon$ we expect a saturation of the error due to the fixed discretization.
In the case of $c_\text{hat}$ the behavior of the $L^2$-error in \autoref{abb:vergleich:konvergenzcg} already hints at this effect for $\epsilon \leq 10^{-3}$.

\begin{table}[htbp]
  \centering
  \begin{tabularx}{0.87\textwidth}{ccccc}\toprule
     & \multicolumn{2}{c}{$c_{\text{hat}}$} & \multicolumn{2}{c}{$c_{\text{plateau}}$} \\
     & $L^2([0,T], L^2)$ & $H^2([0,T], L^2)$ & $L^2([0,T], L^2)$ & $H^2([0,T], L^2)$ \\ \midrule
    \multirow{1}{*}{$n=1$} & $0.25$ & $0.00$ & $0.29$ & $0.30$ \\
      \multirow{1}{*}{$n=2$} & $0.31$ & $0.07$ & $0.25$ & $0.22$ \\
      \multirow{1}{*}{$n=3$} & $0.32$ & $0.08$ & $0.33$ & $0.29$ \\
      \bottomrule
  \end{tabularx}
  \caption{Numerically observed orders of convergence when reconstructing from $u^\epsilon$.}\label{tab:vergleich:konvergenz}
\end{table}

Now we turn to reconstructing $c$ from incomplete noisy measurements $g^\epsilon \approx \Phi c \in (\R^l)^d$, considering two measurement setups. The first one consists of $5^n$ sensors which are grid-like distributed in $\Omega$, as shown in Figure~\ref{abb:vergleich:gebiet}. Each sensor generates measurements for $20$ equidistant times in $(0,T)$. This defines the space-time-positions $(t_i^s, x_i^s)_{i=1,\dots, l}$ of $l=20\cdot 5^n$ measurement points.
We furthermore set $r_x \assign 0.05$ and $r_t\assign 0.02$.

In a real application it might be impossible (or inaffordable) to fill the whole domain with sensors.
For $n=2$ we simulate this by placing the $25$ sensors on the left and lower edges of the domain. To avoid overlap we reduce the sensor radius in this case from $0.05$ to $0.035$.

\begin{figure}[tb]
  \centering
    \subfloat[Grid-shaped arrangement.]{
      \includegraphics{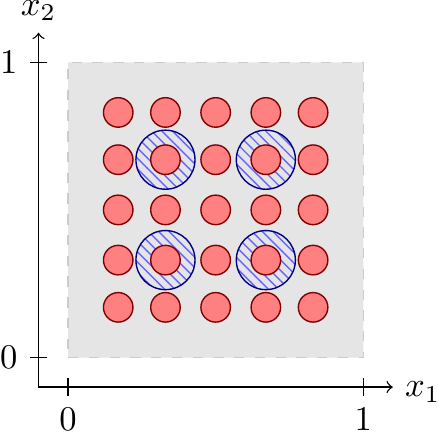}\label{abb:vergleich:gebiet}
    }
    \subfloat[L-shaped arrangement.]{
      \includegraphics{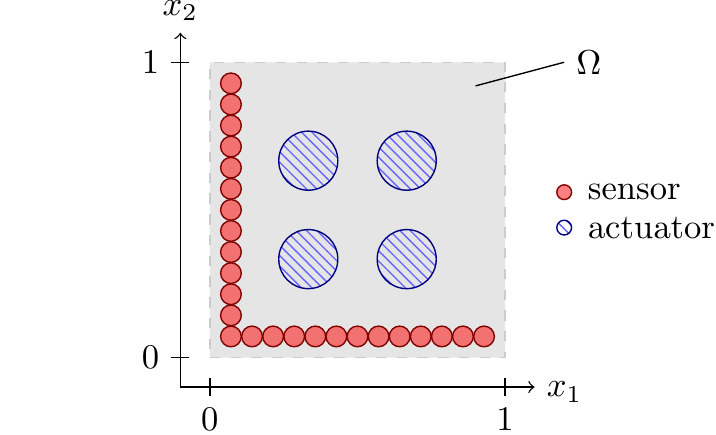}\label{abb:vergleich:gebietl}
    }
  \caption{Distribution of $4$ actuators and $25$ sensors in $\Omega=(0,1)^2$.}
\end{figure}

The reconstructions for $n=2$ from $2000$ values ($20\cdot 5^2$ values for each of the $2^2$ right-hand sides) in the grid-like setting with $1\%$ artificial noise are shown in \autoref{abb:vergleich:paramsrekonstrg}.
They look very similar to the reconstructions in \autoref{abb:vergleich:paramsrekonstru}, where the whole wave (about $1.6 \cdot 10^6$ degrees of freedom) was available.
The errors are listed in \autoref{tab:vergleich:paramsrekonstrg} and most of them are only slightly higher than the corresponding entry of \autoref{tab:vergleich:paramsrekonstru}. For $1\%$ noise this measurement setup seems to be sufficient to obtain roughly the same reconstruction quality as from $u^\epsilon$.

\begin{figure}[!htb]
  \centering
  \subfloat[$c_{k^*}(1.0, \cdot)$ for $c_{\text{hat}}$.]{
    \includegraphics{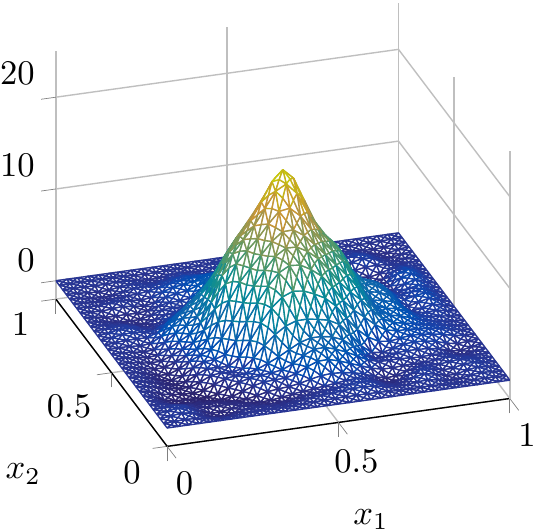}
  }
  \subfloat[$c_{k^*}(1.7, \cdot)$ for $c_{\text{hat}}$.]{
    \includegraphics{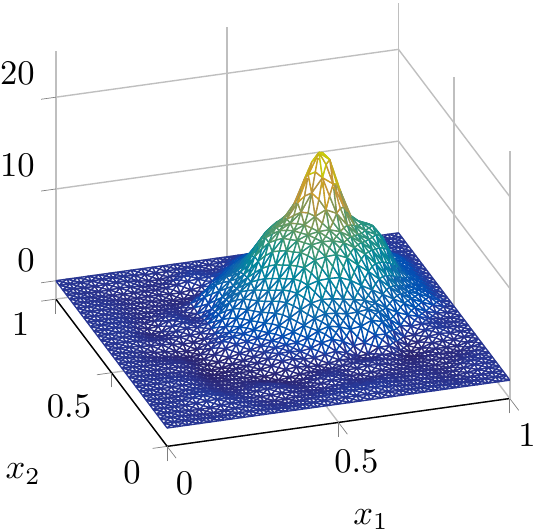}
  }

  \subfloat[$c_{k^*}(1.0, \cdot)$ for $c_{\text{plateau}}$.]{
    \includegraphics{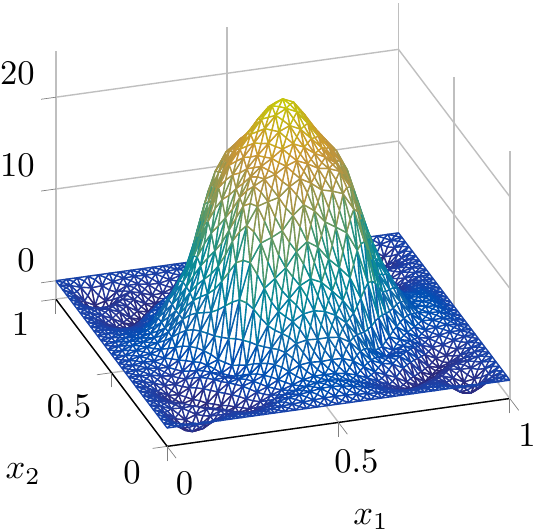}
  }
  \subfloat[$c_{k^*}(1.7, \cdot)$ for $c_{\text{plateau}}$.]{
    \includegraphics{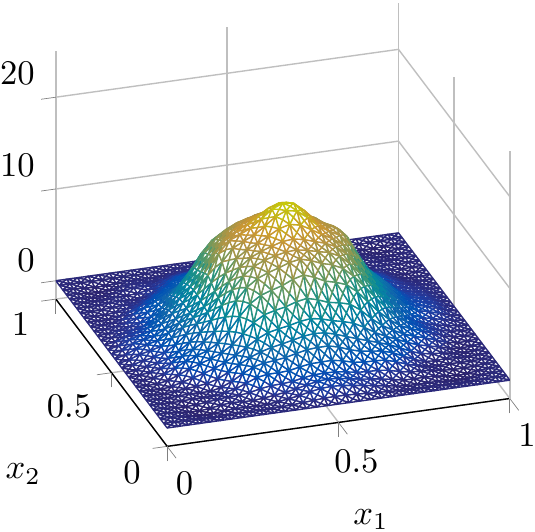}
  }
  \caption{Reconstructions from $g^\epsilon$ with $\epsilon=10^{-2}$ for $n=2$.}\label{abb:vergleich:paramsrekonstrg}
\end{figure}

\begin{table}[tb]
  \centering
  \begin{tabularx}{0.87\textwidth}{Xcccc}\toprule
     & \multicolumn{2}{c}{$c_{\text{hat}}$} & \multicolumn{2}{c}{$c_{\text{plateau}}$} \\
     & $L^2([0,T], L^2)$ & $H^2([0,T], L^2)$ & $L^2([0,T], L^2)$ & $H^2([0,T], L^2)$ \\ \midrule
    $n=1$ & $26.55 \%$ & $58.74 \%$ & $27.28 \%$ & $43.67 \%$ \\
    $n=2$ & $46.11 \%$ & $73.02 \%$ & $34.48 \%$ & $39.24 \%$ \\
    $n=3$ & $66.17 \%$ & $84.72 \%$ & $49.30 \%$ & $51.49 \%$ \\
    \bottomrule
  \end{tabularx}
  \caption{Errors of the reconstruction from $g^\epsilon$ for $\epsilon=10^{-2}$.}\label{tab:vergleich:paramsrekonstrg}
\end{table}

When repositioning the sensors as shown in Figure~\ref{abb:vergleich:gebietl} the quality of the reconstruction decreases, as can be seen in \autoref{abb:vergleich:rekonstrl}. The reconstruction of the moving hat $c_{\text{hat}}$ only achieves an $L^2([0,T], L^2)$ error of $66.29\%$, which is significantly higher than the error of $46.11\%$ when using the $5\times 5$ grid.
For $c_{\text{plateau}}$ the results are more encouraging, $41.79\%$ compared to $34.48\%$.
The reconstruction quality suffers in particular for $t>1.5$, when the reconstruction vanishes in a neighborhood of the corner $(1,1)$ farthest from the sensors. This is of course a consequence of the finite speed of propagation.
For smaller $\epsilon = 10^{-4}$, $c_{\text{plateau}}$ is better approximated ($26.9\%$ error), but the error for $c_{\text{hat}}$ remains at $53.6\%$.

\begin{figure}[htb]%
  \centering
  \subfloat[$c_{k^*}(1.0, \cdot)$ for $c_{\text{hat}}$.]{
    \includegraphics{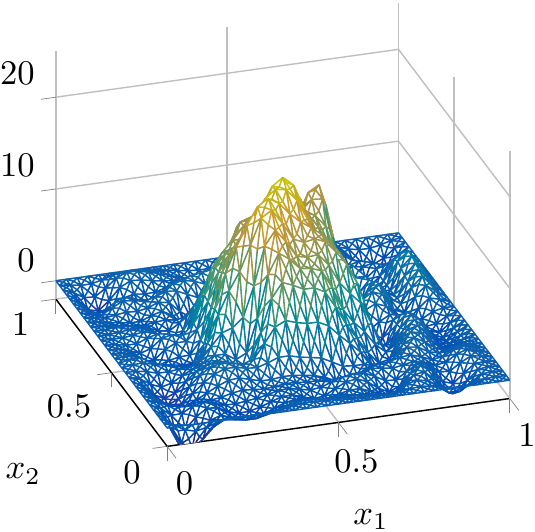}
  }
  \subfloat[$c_{k^*}(1.7, \cdot)$ for $c_{\text{hat}}$.]{
    \includegraphics{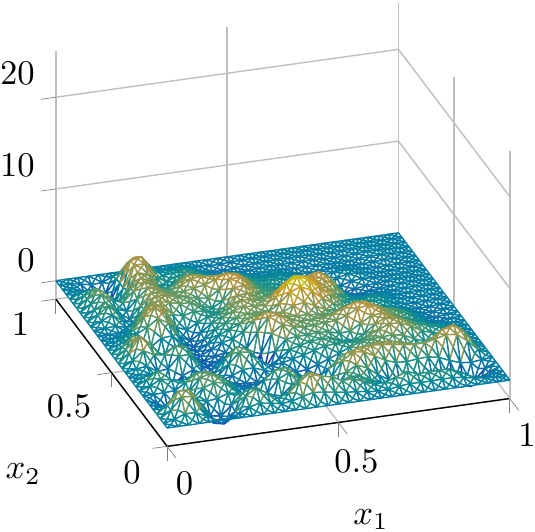}
  }

  \subfloat[$c_{k^*}(1.0, \cdot)$ for $c_{\text{plateau}}$.]{
    \includegraphics{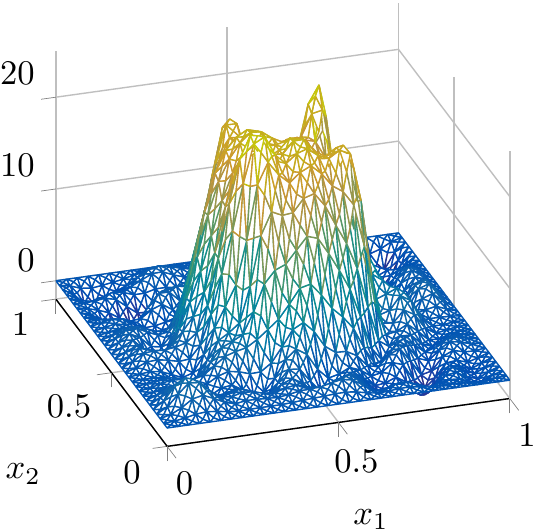}
  }
  \subfloat[$c_{k^*}(1.7, \cdot)$ for $c_{\text{plateau}}$.]{
    \includegraphics{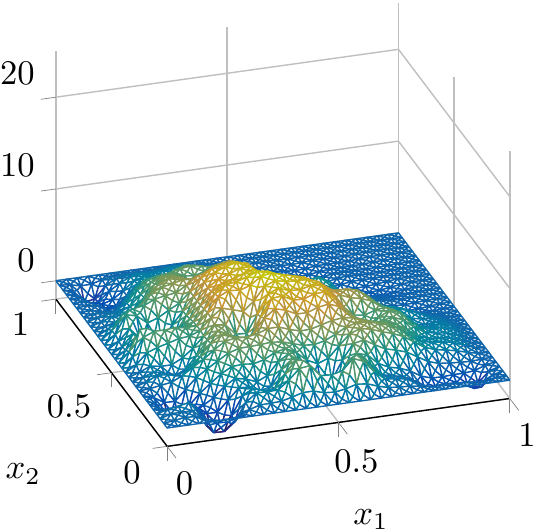}
  }
  \caption{Reconstructions from $g^\epsilon$ with $\epsilon=10^{-2}$ for $n=2$ when the sensors are distributed in an L-shape.}\label{abb:vergleich:rekonstrl}
\end{figure}%

We wish to remark that the discretization was chosen in such a way that computations for $n=3$ and small $\epsilon$ can be done in affordable time. Typical computing times range for $\epsilon = 10^{-2}$ ranged from $5$ seconds ($n=1$) to $25$ minutes ($n=3$) when measured on an Intel i7--2600 CPU and required between $50$ MiB and $750$ MiB of memory.

\FloatBarrier%
\printbibliography%

\end{document}